\documentclass{article}

\usepackage{arxiv}

\usepackage[square,numbers]{natbib}

\usepackage[utf8]{inputenc} 
\usepackage[T1]{fontenc}    
\usepackage{url}            
\usepackage{booktabs}       
\usepackage{amsfonts}       
\usepackage{nicefrac}       
\usepackage{microtype}      
\usepackage{xcolor}         

\usepackage{times}

\usepackage{lipsum}
\usepackage{graphicx}

\usepackage{mathrsfs}
\usepackage{xargs}
\usepackage{enumitem}

\usepackage{aliascnt}
\usepackage{appendix}
\usepackage[english]{babel}

\RequirePackage[colorlinks, linkcolor=blue, citecolor=blue,urlcolor=blue]{hyperref}
\usepackage{amsmath}
\usepackage{amsthm}
\usepackage{cleveref}
\usepackage{amssymb}
\usepackage{caption}
\usepackage{array}
\usepackage{framed}
\usepackage{multirow}

\usepackage{bbm}
\usepackage[ruled,vlined]{algorithm2e}

\title{Mixture weights optimisation for Alpha-Divergence Variational Inference}

\newtheorem{thm}{Theorem}

\crefname{theorem}{Theorem}{Theorems}
\Crefname{theorem}{Theorem}{Theorems}
\crefname{corollary}{Corollary}{Corollaries}
\Crefname{corollary}{Corollary}{Corollaries}
\Crefname{chapter}{Chapter}{Chapters}
\crefname{chapter}{Chapter}{Chapters}

\newtheorem{ex}{Example}
\crefname{example}{example}{examples}
\Crefname{example}{Example}{Examples}

\newaliascnt{lemma}{index}
\newtheorem{lem}[lemma]{Lemma}
\aliascntresetthe{lemma}
\crefname{lemma}{Lemma}{lemmas}
\Crefname{Lemma}{Lemma}{Lemmas}

\newaliascnt{proposition}{index}
\newtheorem{prop}[proposition]{Proposition}
\aliascntresetthe{proposition}
\crefname{proposition}{Proposition}{Propositions}
\Crefname{Proposition}{Proposition}{Propositions}

\newaliascnt{corollary}{index}

\aliascntresetthe{corollary}
\crefname{corollary}{corollary}{corollaries}
\Crefname{Corollary}{Corollary}{Corollaries}

\newaliascnt{definition}{index}

\aliascntresetthe{definition}
\crefname{definition}{definition}{definitions}
\Crefname{Definition}{Definition}{Definitions}

\newaliascnt{remark}{index}

\aliascntresetthe{remark}
\crefname{remark}{remark}{remarks}
\Crefname{remark}{Remark}{Remarks}

\newenvironment{enumerateList}{ \begin{enumerate}[label=(\roman*),wide=0pt, labelindent=\parindent]}{\end{enumerate}}

\newenvironment{hyp}[1]{
\begin{enumerate}[label=\textbf{\sf(#1\arabic*)},resume=hyp#1]\begin{sf}}
{\end{sf}\end{enumerate}}
\crefname{hyp}{}{ass}
\Crefname{hyp}{}{Ass}

\newenvironment{hypD}[1]{
\begin{enumerate}[label=\textbf{\sf(#1\arabic*)},resume=hyp#1]\begin{sf}}
{\end{sf}\end{enumerate}}

\title{Mixture weights optimisation for Alpha-Divergence Variational Inference}

\author{
  Kam\'elia~Daudel \\
LTCI, Télécom Paris \\
Institut Polytechnique de Paris, France \\
\texttt{kamelia.daudel@gmail.com} \\
 \And
Randal~Douc \\
SAMOVAR, Télécom SudParis \\
Institut Polytechnique de Paris, France \\
\texttt{randal.douc@telecom-sudparis.eu}}

\newcommandx{\admiss}[1][1=f]{\mathsf{A}_{#1}}
\newcommandx{\amu}[1][1=\mu_n]{a_{#1}}
\newcommandx{\hamu}[1][1=\mu_n]{\hat a_{#1}}
\newcommand{\arginf}{\mathrm{arginf}}

\newcommand{\argmax}{\mathrm{argmax}}

\newcommandx{\aux}[3][1=\mu, 2=\cte, 3=\alpha]{h_{#1, #2}^{(\alpha)}}

\newcommandx{\binfty}[1][1=\alpha]{|B|_{\infty, #1}}
\newcommandx{\balpha}[1][1=\alpha]{|b|_{#1}}
\newcommandx{\blbd}[2][1=(\lbd{}), 2=j]{ {b_{#2, \alpha}{#1}}}
\newcommandx{\bmuf}[2][1=\mu, 2=\alpha]{ {b_{#1, #2}}}
\newcommandx{\bmufk}[2][1=\mu, 2=\alpha]{\hat{b}_{#1, #2, M}}

\newcommandx{\couple}[2][1=\PQ, 2=\PP]{(#1 || #2)}
\newcommand{\Cov}{\mathbb{C}\mathrm{ov}}
\newcommand{\cte}{\kappa}
\newcommandx{\cteaux}[1][1=\alpha]{L_{#1, 4}}
\newcommandx{\cteinf}[1][1=\alpha]{L_{#1, 2}}
\newcommandx{\ctemono}[1][1=\alpha]{L_{#1, 1}}
\newcommandx{\ctestar}[1][1=\alpha]{L_{#1, 5}}
\newcommandx{\ctesup}[1][1=\alpha]{L_{#1, 3}}
\newcommand{\data}{\mathscr{D}}

\newcommandx{\diverg}[1][1=\alpha]{D_{#1}}
\newcommandx{\divergR}[1][1=\alpha]{D^{(\mathrm{AR})}_{#1}}
\newcommandx{\Domain}[1][1=\alpha]{\mathrm{Dom}_{#1}}

\newcommand{\eqdef}{:=}

\newcommand{\eqsp}{\;}

\newcommandx{\falpha}[1][1=\alpha]{f_{#1}}

\newcommandx{\aei}[1][1=\alpha]{$(#1, \Gamma)$-}
\newcommandx{\GammaAlpha}[1][1=\alpha]{ \Gamma}
\renewcommand{\geq}{\geqslant}
\newcommandx{\gmuf}[1][1=\mu]{ {g_{#1}}}

\newcommandx{\iteration}[1][1=\alpha]{\mathcal{I}_{#1}}
\newcommandx{\iterationK}[1][1=\alpha]{\hat{\mathcal{I}}_{#1, M}}

\newcommandx{\lbd}[2][1=]{
    \ifthenelse{\equal{#1}{}}
    {{\boldsymbol{\lambda}_{#2}}}
    {\lambda_{#1,#2}}
    }
\newcommandx{\lbdp}[2][1=]{
    \ifthenelse{\equal{#1}{}}
    {{\boldsymbol{\lambda'_{#2}}}}
    {\lambda'_{#1,#2}}
    }
\newcommandx{\lbdpp}[2][1=]{
    \ifthenelse{\equal{#1}{}}
    {{\boldsymbol{\lambda''_{#2}}}}
    {\lambda''_{#1,#2}}
    }

\renewcommand{\leq}{\leqslant}

\newcommand{\lr}[1]{\left(#1 \right)}
\newcommand{\lrb}[1]{\left[#1 \right]}
\newcommand{\lrc}[1]{\left\{#1 \right\}}
\newcommand{\lrcb}[1]{\left\{#1 \right\}}

\newcommand{\meas}[1]{\mathrm{M}_{#1}}
\newcommandx{\mixture}[2][1=(y), 2=\lbd{}]{q_{#2, \Theta}#1}
\newcommand{\muf}{{\mu^\star}}
\newcommandx{\mulbd}[1][1=\lbd{}]{{\mu_{#1, \Theta}}}

\newcommandx{\norm}[2][1=\infty]{\|#2\|_{#1}}
\newcommand{\nset}{\mathbb N}
\newcommand{\nstar}{\mathbb{N}^\star}
\newcommand{\PE}{\mathbb E}
\newcommandx{\posterior}[1][1=y]{p(#1|\mathscr{D})}
\newcommand{\PP}{\mathbb P}
\newcommand{\PQ}{\mathbb Q}
\newcommandx{\Psif}[1][1=\alpha]{\Psi_{#1}}
\newcommandx{\PsifAR}[1][1=\alpha]{\Psi_{#1}^{AR}}

\newcommandx{\respa}[2][1=y, 2=\lbd{}]{\hat\gamma_{j}^{(t)} (#1; #2)}
\newcommandx{\respat}[3][1=y, 2=\alpha, 3=j]{\gamma_{#3, #2}^{t} (#1)}
\renewcommand{\rho}{\varrho}
\newcommand{\rmd}{\mathrm d}

\newcommand{\Rset}{\mathbb{R}}
\newcommand{\rset}{\mathbb{R}}

\newcommand{\simplex}{\mathcal{S}}
\newcommand{\tgamma}{{\tilde{\Gamma}}}
\newcommandx{\thetat}[2][1=j, 2=t]{\theta_{#1,#2}}
\newcommandx{\thetav}[2][1=]{
    \ifthenelse{\equal{#1}{}}
    {{\boldsymbol{\theta}_{#2}}}
    {\lambda_{#1,#2}}
    }

\newcommand{\Tset}{{\mathsf{T}}}
\newcommand{\Tsigma}{\mathcal{T}}

\newcommand{\Yset}{\mathsf Y}
\newcommand{\Ysigma}{\mathcal Y}
\newcommand{\Var}{\mathbb{V}\mathrm{ar}}

\begin{document}

\maketitle

\begin{abstract}%
This paper focuses on gradient-based Variational Inference for $\alpha$-divergence minimisation. More precisely, we are interested in studying algorithms making it possible to optimise the mixture weights of any given mixture model, without any information on the underlying distribution of its mixture components parameters. The Power Descent is one such algorithm and we establish in our work its
convergence towards the optimal mixture weights when $\alpha <1$ under alleviated assumptions. We also investigate the link between Power Descent and Entropic Mirror Descent: this allows us to introduce the Renyi Descent, for which we prove an $O(1/N)$ convergence rate. We then provide some numerical experiments to illustrate the behavior of these two algorithms in practice.
\end{abstract}

\section{Introduction}

Bayesian Inference involves being able to compute or sample from the posterior density. For many useful models, the posterior density can only be evaluated up to a normalisation constant and we must resort to approximation methods.

One major category of approximation methods is Variational Inference, a wide class of optimisation methods which introduce a simpler density family $\mathcal{Q}$ and use it to approximate the posterior density (see for example Variational Bayes \citep{Jordan1999, beal.phd} and Stochastic Variational Inference \citep{JMLR:v14:hoffman13a}). The crux of these methods consists in being able to find the best approximation of the posterior density among the family $\mathcal{Q}$ in the sense of a certain divergence, most typically the Kullback-Leibler divergence. However, The Kullback-Leibler divergence is known to have some undesirable properties (e.g posterior overestimation/underestimation \citep{divergence-measures-and-message-passing}) and as a consequence, the $\alpha$-divergence \citep{ZhuRohwer, zhu-rohwer-alpha-div} and Renyi's $\alpha$-divergence \citep{renyi1961, 2012arXiv1206.2459V} have gained a lot of attention recently as a more general alternative \citep{2015arXiv151103243H, 2016arXiv160202311L, NIPS2017_6866, NIPS2017_14e422f0, NIPS2017_7093, NIPS2018_7816, biasedDomke2020, geffner2020difficulty, daudel2020infinitedimensional, daudel2021monotonic, dhaka2021challenges}.

Noticeably, \cite{daudel2020infinitedimensional} introduced the \aei descent, a general family of gradient-based algorithms that are able to optimise the \textit{mixture weights} of mixture models by $\alpha$-divergence minimisation, without any information on the underlying distribution of its mixture components parameters. The benefit of these types of algorithms is that they allow, in an Sequential Monte Carlo fashion \citep{articleDoucet}, to select the mixture components according to their overall importance in the set of component parameters. From there, one is able to optimise the weights and the components parameters alternatively \citep{daudel2020infinitedimensional}. The \aei descent framework recovers the Entropic Mirror Descent algorithm (corresponding to $\GammaAlpha(v) = e^{-\eta v}$ with $\eta > 0$) and includes the Power Descent, an algorithm defined for all $\alpha \in \Rset \setminus \lrcb{1}$ and all $\eta > 0$ that sets $\GammaAlpha(v) = [(\alpha-1)v + 1]^{\eta/(1-\alpha)}$. Although these two algorithms are linked to one another from a theoretical perspective through the \aei descent framework, numerical experiments in \cite{daudel2020infinitedimensional} showed that the Power Descent outperforms the Entropic Mirror Descent when $\alpha <1$ as the dimension increases.

Nonetheless, the global convergence of the Power Descent algorithm when $\alpha <1$, as stated in \cite{daudel2020infinitedimensional}, is subjected to the condition that the limit exists. Furthermore, even though the convergence towards the global optimum is derived, there is no convergence rate available for the Power Descent when $\alpha <1$. 
While there is no general rule yet on how to select the value of $\alpha$ in practice, the case $\alpha < 1$ has the advantage that it enforces a \textit{mass-covering} property, as opposed to the \textit{mode-seeking} property exhibited when $\alpha \geq 1$ (\cite{divergence-measures-and-message-passing} and \cite{daudel2020infinitedimensional}) and which often may lead to posterior variance underestimation. We are thus interested in studying Variational Inference methods for optimising the mixture weights of mixture models when $\alpha <1$. After recalling the basics of the Power Descent algorithm in \Cref{sec:optimPb}, we make the following contributions in the paper:
\begin{itemize}[label=$\bullet$ ,wide=0pt, labelindent=\parindent]
\item In \Cref{sec:cv}, we derive the full convergence proof of the Power Descent algorithm towards the optimal mixture weights when $\alpha < 1$ (\Cref{thm:limit}).

\item Since the $\alpha$-divergence becomes the traditional forward Kullback-Leibler when $\alpha \to 1$, we first bridge in \Cref{sec:EMD} the gap between the cases $\alpha <1$ and $\alpha > 1$ of the Power Descent: we obtain that the Power Descent recovers an Entropic Mirror Descent performing forward Kullback-Leibler minimisation (\Cref{lem:extenstionAlpha1}). We then keep on investigating the connections between the Power Descent and the Entropic Mirror Descent by considering first-order approximations. In doing so, we are able to go beyond the \aei descent framework and to introduce an algorithm closely-related to the Power Descent that we call the \textit{Renyi Descent} and that is proved in \Cref{thm:Renyi} to converge at an $O(1/N)$ rate towards its optimum for all $\alpha \in \rset$.

\item Finally, we run some numerical experiments in \Cref{sec:numerical} to compare the behavior of the Power Descent and the Renyi Descent altogether, before discussing the potential benefits of one approach over the other.
\end{itemize}

\section{Background}
\label{sec:optimPb}

We start by introducing some notation. Let $(\Yset,\Ysigma, \nu)$ be a measured space, where $\nu$ is a $\sigma$-finite measure on $(\Yset, \Ysigma)$. Assume that we have access to some observed variables $\data$ generated from a probabilistic model $p(\data|y)$ parameterised by a hidden random variable $y \in \Yset$ that is drawn from a certain prior $p_0(y)$. The posterior density of the latent variable $y$ given the data $\data$ is then given by:
$$
p(y | \data) = \frac{p(y, \data)}{p(\data)} = \frac{ p_0(y) p(\data|y)}{p(\data)} \eqsp,
$$
where the normalisation constant $p(\data) = \int_\Yset p_0(y) p(\data|y) \nu(\rmd y)$ is called the {\em marginal likelihood} or {\em model evidence} and is oftentimes unknown.

To approximate the posterior density, the Power Descent considers a variational family $\mathcal{Q}$ that is large enough to contain mixture models and that we redefine now: letting $(\Tset,\Tsigma)$ be a measurable space, $K:(\theta,A) \mapsto \int_A k(\theta,y)\nu(\rmd y)$ be a Markov transition kernel on $\Tset \times \Ysigma$ with kernel density $k$ defined on $\Tset \times \Yset$, the Power Descent considers the following approximating family
$$
\lrc{y \mapsto \int_\Tset \mu(\rmd \theta) k(\theta,y) \eqsp: \eqsp \mu \in \mathsf M} \eqsp,
$$
where $\mathsf{M}$ is a convenient subset of $\meas{1}(\Tset)$, the set of probability measures on $(\Tset, \Tsigma)$. This choice of approximating family extends the typical parametric family commonly-used in Variational Inference since it amounts to putting a prior over the parameter $\theta$ (in the form of a measure) and does describe the class of mixture models when $\mu$ is a weighted sum of Dirac measures.

\paragraph{Problem statement} Denote by $\PP$ the probability measure on $(\Yset, \Ysigma)$ with corresponding density $p(\cdot|\data)$ with respect to $\nu$ and for all $\mu\in\meas{1}(\Tset)$, for all $y \in \Yset$, denote $\mu k(y) = \int_\Tset \mu(\rmd \theta) k(\theta, y)$. Furthermore, given $\alpha \in \rset$, let $\falpha$ be the convex function on $(0, +\infty)$ defined by $\falpha[0](u) = u -1 -\log(u)$, $\falpha[1](u) = 1- u + u\log(u)$ and $\falpha(u) = \frac{1}{\alpha(\alpha-1)} \left[  u^\alpha -1 -\alpha(u-1) \right]$ for all $\alpha \in \rset \setminus \lrcb{0,1}$. Then, the $\alpha$-divergence between $\mu K$ and $\PP$ (extended by continuity to the cases $\alpha = 0$ and $\alpha = 1$ as for example done in \cite{alpha-beta-gamma}) is given by
\begin{align*}
\diverg\couple[\mu K][\PP] = \int_\Yset \falpha\left(\frac{\mu k(y)}{p(y|\data)} \right) p(y|\data) \nu(\rmd y) \eqsp,
\end{align*}
and the goal of the Power Descent is to find
\begin{align}\label{eq:optim:alpha}
\arginf_{\mu \in \mathsf{M}} \diverg\couple[\mu K][\PP] \eqsp.
\end{align}
More generally, letting $p$ be any measurable positive function on $(\Yset,\Ysigma)$, the Power Descent aims at solving
\begin{align}\label{eq:GeneralProblem}
\arginf_{\mu \in \mathsf{M}} \Psif(\mu; p)\eqsp,
\end{align}
where for all $\mu \in \meas{1}(\Tset)$, $\Psif(\mu; p) = \int_\Yset \falpha \left(\mu k(y)/ p(y) \right) p(y) \nu (\rmd y)$.
The Variational Inference optimisation problem \eqref{eq:optim:alpha} can then be seen as an instance of \eqref{eq:GeneralProblem} that is equivalent to optimising $\Psif(\mu; p)$ with $p(y) = p(y, \data)$ (see \Cref{subsec:equi}). In the following, the dependency on $p$ in $\Psif$ may be dropped throughout the paper for notational ease when no ambiguity occurs and we now present the Power Descent algorithm.



\paragraph{The Power Descent algorithm.} The optimisation problem \eqref{eq:GeneralProblem} can be solved for all $\alpha \in \rset \setminus \lrcb{1}$ by using the Power Descent algorithm introduced in \cite{daudel2020infinitedimensional} : given an initial measure $\mu_1 \in\meas{1}(\Tset)$ such that $\Psif(\mu_1) < \infty$, $\alpha \in \rset \setminus \lrcb{1}$, $\eta > 0$ and $\cte$ such that $(\alpha-1)\cte \geq 0$, the Power descent algorithm is an iterative scheme which builds the sequence of probability measures $(\mu_n)_{n\in\nstar}$
\begin{equation}
\label{eq:def:mu}
\mu_{n+1}= \iteration (\mu_n)\;,\qquad n\in\nset^\star \eqsp,
\end{equation}
where for all $\mu \in \meas{1}(\Tset)$, the one-step transition $\mu \mapsto \iteration(\mu)$ is given by Algorithm \ref{algo:aei} and where for all $v\in \Domain$,
 $\GammaAlpha(v) = [(\alpha-1)v +1]^{\eta/(1-\alpha)}$ [and $\Domain$ denotes an interval of $\rset$ such that for all $\theta \in \Tset$, all $\mu \in \meas{1}(\Tset)$, $\bmuf(\theta) + \cte$ and $\mu(\bmuf) + \cte \in \Domain$].

\begin{algorithm}[H]
\caption{\em Power descent one-step transition $(\GammaAlpha(v) = [(\alpha-1)v +1]^{\eta/(1-\alpha)})$}
\label{algo:aei}
\begin{enumerate}
\item \underline{Expectation step} : \setlength{\parindent}{1cm} $\bmuf(\theta)= \mathlarger{\int_\Yset} k(\theta,y)  \falpha' \left(\dfrac{\mu k(y)}{p(y)}\right) \nu(\rmd y)$
\item \underline{Iteration step} : \setlength{\parindent}{1cm} $\iteration (\mu)(\rmd \theta)= \dfrac{\mu(\rmd \theta) \cdot \GammaAlpha(\bmuf(\theta)+\cte)}{\mu(\GammaAlpha(\bmuf+\cte))}$
\end{enumerate} \
\end{algorithm}

In this algorithm, $\bmuf$ can be understood as the gradient of $\Psif$. Algorithm \ref{algo:aei} then consists in applying the transform function $\GammaAlpha$ to the translated gradient $\bmuf + \cte$ and projecting back onto the space of probability measures. 

A remarkable property of the Power Descent algorithm, which has been proven in \cite{daudel2020infinitedimensional} (it is a special case of \cite[Theorem 1]{daudel2020infinitedimensional} with $\GammaAlpha(v) = [(\alpha-1)v +1]^{\eta/(1-\alpha)}$), is that under \ref{hyp:positive} as defined below
 \begin{hyp}{A}
  \item \label{hyp:positive} The density kernel $k$ on $\Tset\times\Yset$, the function $p$ on $\Yset$ and
  the $\sigma$-finite measure $\nu$ on $(\Yset,\Ysigma)$ satisfy, for all
  $(\theta,y) \in \Tset \times \Yset$, $k(\theta,y)>0$, $p(y)>0$ and
  $\int_\Yset p(y) \nu(\rmd y)<\infty$.
\end{hyp}
the Power Descent ensures a monotonic decrease in the $\alpha$-divergence at each step for all $\eta \in (0,1]$ (this result is recalled in \Cref{thm:admiss} of \Cref{subsec:lem:admiss} for the sake of completeness). Under the additional assumptions that $\cte > 0$ and
\begin{align}\label{eq:condCV}
\sup_{\theta \in \Tset, \mu \in \meas{1}(\Tset)} |\bmuf| < \infty \quad \mbox{and} \quad \Psif(\mu_1) < \infty \eqsp,
\end{align}
the Power Descent is also known to converge towards its optimal value at an $O(1/N)$ rate  when $\alpha >1$ \citep[Theorem 3]{daudel2020infinitedimensional}. On the other hand, when $\alpha < 1$, the convergence towards the optimum as written in \cite{daudel2020infinitedimensional} holds under different assumptions including
\begin{hyp}{A}
  \item\label{hyp:compactFull}

\begin{enumerate}[label=(\roman*)]
\item \label{item:theta:oneFull}$\Tset$ is a compact metric space and $\Tsigma$ is
  the associated Borel $\sigma$-field;
\item \label{item:theta:twoFull} for all $y \in \Yset$, $\theta \mapsto k(\theta,y)$ is continuous;
\item \label{item:theta:fourFull} we have $\int_\Yset {\sup_{\theta \in \Tset}k(\theta,y)} \times \sup_{\theta' \in \Tset} \lr{\frac{k(\theta',y)}{p(y)}}^{\alpha-1} \nu(\rmd y)<\infty$.
\end{enumerate}
If $\alpha = 0$, assume in addition that $\int_\Yset {\sup_{\theta \in \Tset} \left| \log\lr{\frac{k(\theta,y)}{p(y)}} \right|} p(y) \nu(\rmd y)<\infty$.
\end{hyp}
so that \cite[Theorem 4]{daudel2020infinitedimensional}, that is recalled below under the form of \Cref{thm:repulsive}, states the convergence of the Power Descent algorithm towards the global optimum.
\begin{thm}[{\cite[Theorem 4]{daudel2020infinitedimensional}}]
  \label{thm:repulsive} Assume \ref{hyp:positive} and
  \ref{hyp:compactFull}. Let $\alpha < 1$ and let $\cte \leq 0$. Then, for all
  $\mu\in\meas{1}(\Tset)$,  $\Psif(\mu) < \infty$ and any $\eta > 0$ satisfies $0 <\mu(\GammaAlpha(\bmuf + \cte))<\infty$. Further assume that $\eta \in (0,1]$ and that there exist $\mu_1,\muf \in\meas{1}(\Tset)$ such that the (well-defined) sequence $(\mu_n)_{n\in\nstar}$ defined by \eqref{eq:def:mu} weakly converges to $\muf$ as $n\to\infty$. Finally, denote by $\meas{1, \mu_1}(\Tset)$ the set of probability measures dominated by $\mu_1$. Then the following assertions hold
\begin{enumerate}[label=(\roman*)]
\item \label{item:rep1bis} $(\Psif(\mu_n))_{n\in\nstar}$ is nonincreasing,
\item\label{item:rep1} $\muf$ is a fixed point of $\iteration$,
\item\label{item:rep2} $\Psif(\muf)=\inf_{\zeta \in \meas{1,\mu_1}(\Tset)} \Psif(\zeta)$.
\end{enumerate}
\end{thm}
The above result assumes there must exist $\mu_1,\muf \in\meas{1}(\Tset)$ such that the sequence $(\mu_n)_{n\in\nstar}$ defined by \eqref{eq:def:mu} weakly converges to $\muf$ as $n\to\infty$, that is it assumes the limit already exists. Our first contribution consists in showing that this assumption can be alleviated when $\mu$ is chosen a weighted sum of Dirac measures, that is when we seek to perform mixture weights optimisation by $\alpha$-divergence minimisation. 

\section{Convergence of the Power Descent algorithm in the mixture case}
\label{sec:cv}

Before we state our convergence result, let us first make two comments on the assumptions from \Cref{thm:repulsive} that shall be retained in our upcoming convergence result.

A first comment is that \ref{hyp:positive} is mild since the assumption that $p(y) > 0$ for all $y \in \Yset$ can be discarded and is kept for convenience \cite[Remark 4]{daudel2020infinitedimensional}. A second comment is that \ref{hyp:compactFull} is also mild and covers \eqref{eq:condCV} as it amounts to assuming that $\bmuf(\theta)$ and $\Psif(\mu)$ are uniformly bounded with respect to $\mu$ and $\theta$. To see this, we give below an example for which \ref{hyp:compactFull} is satisfied.

\begin{ex}\label{ex:thm:limit} Consider the case $\Yset = \rset^d$ with $\alpha \in [0,1)$. Let $r> 0$ and let $\Tset = \mathcal{B}(0,r) \subset \rset^d$. Furtheremore, let $K_{h}$ be a Gaussian transition kernel with bandwidth $h$ and denote by $k_{h}$ its associated kernel density. Finally, let $p$ be a mixture density of two $d$-dimensional Gaussian distributions multiplied by a positive constant $c$ such that $p(y) = c \times \lrb{ 0.5 \mathcal{N}(y; \theta^\star_1, \boldsymbol{I_d}) + 0.5 \mathcal{N}(y; \theta^\star_2 , \boldsymbol{I_d})}$ for all $y \in \Yset$ where $\theta_1^\star, \theta_2^\star \in \Tset$ and $\boldsymbol{I_d}$ is the identity matrix. Then, \ref{hyp:compactFull} holds (see \Cref{sec:exGaussian}).
\end{ex}
Next, we introduce some notation that are specific to the case of mixture models we aim at studying in this section. Given $J \in \nstar$, we introduce the simplex of $\rset^J$:
$$
\simplex_J = \lrc{ \lbd{}= (\lambda_1, \ldots, \lambda_J) \in \rset^J \eqsp : \eqsp \forall j \in \lrcb{1, \ldots , J}, \eqsp \lambda_j \geq 0 \eqsp \mbox{and} \eqsp \sum_{j=1}^J \lambda_j = 1} \eqsp
$$
and we also define $\simplex_J^+ = \lrc{ \lbd{} \in \simplex_J \eqsp : \eqsp \forall j \in \lrcb{1, \ldots , J}, \eqsp \lambda_j > 0}$. We let $\Theta = (\theta_1, \ldots, \theta_J) \in \Tset^J$ be fixed and for all $\lbd{} \in \simplex_J$, we define $\mulbd  \in \meas{1}(\Tset)$ by $\mulbd  = \sum_{j=1}^J \lambda_j \delta_{\theta_j}$.

Consequently, $\mulbd  k(y) = \sum_{j = 1}^J \lambda_j k(\theta_{j}, y)$ corresponds to a mixture model and if we let $(\mu_n)_{n\in\nstar}$ be defined by $\mu_1 = \mulbd $ and \eqref{eq:def:mu}, an immediate induction yields that for every $n\in \nstar$, $\mu_n$ can be expressed as $\mu_n=\sum_{j=1}^J \lbd[j]{n} \delta_{\theta_j}$ where $\lbd{n}=(\lbd[1]{n},\ldots,\lbd[J]{n}) \in \simplex_J$ satisfies the  initialisation $\lbd{1}=\lbd{}$ and the update formula:
\begin{align}\label{eq:iteration:mixture}
\lbd{n+1} = \iteration^{\mathrm{mixt}}(\lbd{n}) \eqsp, \eqsp n\in\nstar \eqsp,
\end{align}
where for all $\lbd{} \in \simplex_J$,
$$
\iteration^{\mathrm{mixt}}(\lbd{}) \eqdef \left( \dfrac{\lambda_{j} \GammaAlpha( \bmuf[\mulbd ](\theta_j) + \cte)}{\sum_{\ell=1}^{J} \lambda_{\ell} \GammaAlpha(\bmuf[\mulbd ](\theta_\ell)+ \cte )} \right)_{1 \leq j \leq J}
$$
with $\GammaAlpha(v) = [(\alpha-1)v + 1]^{\frac{\eta}{1-\alpha}}$ for all $v \in \Domain$. Finally, let us rewrite \ref{hyp:compactFull} in the simplified case where $\mu$ is a sum of Dirac measures, which gives \ref{hyp:compact} below.
\begin{hyp}{A}
  \item\label{hyp:compact}
\begin{enumerate}[label=(\roman*)]
\item \label{item:theta:two} For all $y \in \Yset$, $\theta \mapsto k(\theta,y)$ is continuous;
\item \label{item:theta:four} we have $\mathlarger\int_\Yset \max \limits_{1 \leq j \leq J} k(\theta_j ,y) \times  \max \limits_{1 \leq j' \leq J} \lr{ \frac{k(\theta_{j'},y)}{p(y)}}^{\alpha-1}   \nu(\rmd y)<\infty $.
\end{enumerate}
If $\alpha = 0$, we assume in addition that $\mathlarger\int_\Yset \max \limits_{1 \leq j \leq J}  \left|\log \lr{ \frac{k(\theta_{j},y)}{p(y)}} \right| p(y) \nu(\rmd y)<\infty $.
\end{hyp}
We then have the following theorem, which establishes the full proof of the global convergence towards the optimum for the mixture weights when $\alpha <1$.

\begin{thm}
  \label{thm:limit}
  Assume \ref{hyp:positive} and \ref{hyp:compact}. Let $\alpha < 1$, let $\Theta = (\theta_1, \ldots, \theta_J) \in \Tset^J$ be fixed and let $\cte$ be such that $\cte \leq 0$.  Then for all $\lbd{} \in \simplex_J$, $\Psif(\mulbd ) < \infty$ and for any $\eta > 0$ the sequence $(\lbd{n})_{n\in\nset^\star}$ defined by $\lbd{1} \in \simplex_J$ and \eqref{eq:iteration:mixture} is well-defined. If in addition $(\lbd{1},\eta) \in \simplex_J^+ \times (0,1]$ and $\lrcb{K(\theta_1,\cdot),\ldots,K(\theta_J,\cdot)}$ are linearly independent, then
    \begin{enumerate}[label=(\roman*)]
    \item $(\Psif(\mulbd[\lbd{n}]))_{n\in\nstar}$ is nonincreasing,
    \item the sequence $(\lbd{n})_{n\in\nset^\star}$ converges to some $\lbd{\star} \in \simplex_J$ which is a fixed point of $\iteration^{ \mathrm{mixt}}$,
    \item  $\Psif(\mulbd[\lbd{\star}]) = \inf_{\lbd{}' \in \simplex_J} \Psif(\mulbd[\lbd{}'] )$.
    \end{enumerate}
\end{thm}
The proof of this result builds on \Cref{thm:repulsive} and \Cref{thm:admiss} and is deferred to \Cref{subsec:prooflimit}. Notice that since $\Psif$ depends on $\lbd{}$ through $\mulbd{} K$ in \Cref{thm:limit}, an identifiably condition was to be expected in order to achieve the convergence of the sequence $(\lbd{n})_{n\in\nset^\star}$. Following \Cref{ex:thm:limit}, this identifiably condition notably holds for $J \leq d$ under the assumption that the $\theta_1, ..., \theta_J$ are full-rank.

We thus have the convergence of the Power Descent under less stringent conditions when $\alpha <1$ and when we consider the particular case of mixture models. This algorithm can easily become feasible for any choice of kernel $K$ by resorting to an unbiased estimator of $(\bmuf[{\mulbd[\lbd{n}]}](\theta_j))_{1\leq j \leq J}$ in the update formula \eqref{eq:iteration:mixture} (see Algorithm \ref{algo:mixture} of \Cref{subsec:aeiprac}).

Nevertheless, contrary to the case $\alpha >1$ we still do not have a convergence rate for the Power Descent when $\alpha <1$. Furthermore, the important case $\alpha \to 1$, which corresponds to performing forward Kullback-Leibler minimisation, is not covered by the Power Descent algorithm. In the next section, we extend the Power Descent to the case $\alpha = 1$. As we shall see, this will lead us to investigate the connections between the Power Descent and the Entropic Mirror Descent beyond the \aei descent framework. As a result, we will introduce a novel algorithm closely-related to the Power Descent that yields an $O(1/N)$ convergence rate when $\mu = \mulbd{}$ and $\alpha <1$ (and more generally when $\mu \in \meas{1}(\Tset)$ and $\alpha \in \rset$).

\section{Power Descent and Entropic Mirror Descent}
\label{sec:EMD}

Recall from \Cref{sec:optimPb} that the Power Descent is defined for all $\alpha \in \rset \setminus \lrcb{1}$. In this section, we first establish in \Cref{lem:extenstionAlpha1} that the Power Descent can be extended to the case $\alpha = 1$ and that we recover an Entropic Mirror Descent, showing that a deeper connection runs between the two approaches beyond the one identified by the \aei descent framework. This result relies on typical convergence and differentiability assumptions summarised in \ref{hypLimAlpha1} and which are deferred to \Cref{sec:FirstLemma}, alongside with the proof of \Cref{lem:extenstionAlpha1}.

\begin{prop}[Limiting case $\alpha \to 1$] \label{lem:extenstionAlpha1} Assume \ref{hyp:positive} and \ref{hypLimAlpha1}. Then, for all continuous and bounded real-valued functions $h$ on $\Tset$, we have that
$$
\lim_{\alpha\to 1} [\iteration (\mu)](h) = [\iteration[1](\mu)](h) \eqsp,
$$
where for all $\mu \in \meas{1}(\Tset)$ and all $\theta \in \Tset$, we have set
\begin{align}\label{eq:extensionAlpha1}
\iteration[1](\mu)(\rmd \theta) =  \frac{\mu(\rmd \theta) e^{- \eta \bmuf[\mu][1](\theta)}}{ \mu\left(  e^{- \eta \bmuf[\mu][1]} \right) } \quad \mbox{and} \quad  \bmuf[\mu][1](\theta) = \int_\Yset k(\theta,y) \log \left(\frac{\mu k(y)}{p(y)} \right) \nu(\rmd y) \eqsp.
\end{align}
\end{prop}
Here, we recognise the one-step transition associated to the Entropic Mirror Descent applied to $\Psif[1]$. This algorithm is a special case of \cite{daudel2020infinitedimensional} with $\GammaAlpha(v) = e^{-\eta v}$ and $\alpha = 1$ and as such, it is known to lead to a systematic decrease in the forward Kullback-Leibler divergence and to enjoy an $O(1/N)$ convergence rate under the assumptions that \eqref{eq:condCV} holds and $\eta \in (0,1)$ \cite[Theorem 3]{daudel2020infinitedimensional}.

We have thus obtained that the Power Descent coincides exactly with the Entropic Mirror Descent applied to $\Psif[1]$ when $\alpha = 1$ and we now focus on understanding the links between Power Descent and Entropic Mirror Descent when $\alpha \in \rset \setminus \lrcb{1}$. For this purpose, let $\cte$ be such that $(\alpha -1) \cte \geq 0$ and let us study first-order approximations of the Power Descent and the Entropic Mirror Descent applied to $\Psif$ when $\bmuf[\mu_n](\theta) \approx \mu_n(\bmuf[\mu_n])$ for all $\theta \in \Tset$.

Letting $\eta > 0$, we have that the update formula for the Power Descent is given by
$$
\mu_{n+1} (\rmd \theta) = \frac{\mu_n (\rmd \theta)\left[(\alpha-1) (\bmuf[\mu_n](\theta)+\cte) + 1\right]^{\frac{\eta}{1-\alpha}} }{\mu_n(\left[(\alpha-1)(\bmuf[\mu_n]+\cte) +1\right]^{\frac{\eta}{1-\alpha}})} \eqsp, \quad n\in\nset^\star \eqsp .
$$
Now using the first order approximation $u^{\frac{\eta}{1-\alpha}} \approx v^{\frac{\eta}{1-\alpha}} - \frac{\eta}{\alpha-1} v^{\frac{\eta}{1-\alpha}-1}(u -v)$ with $u = \frac{(\alpha-1)(\bmuf[\mu_n](\theta)+\cte) +1}{(\alpha-1)(\mu(\bmuf[\mu_n])+\cte) +1}$ and $v =1$, we can deduce the following approximated update formula
$$
\mu_{n+1}(\rmd \theta) = \mu_n (\rmd \theta)\left[1 - \frac{\eta}{\alpha-1}  \frac{\bmuf[\mu_n](\theta)- \mu_n(\bmuf[\mu_n]) }{\mu_n(\bmuf[\mu_n])+ \cte +1/(\alpha-1)}  \right] \eqsp, \quad n\in\nset^\star \eqsp.
$$
Letting $\eta' >0$, the update formula for the Entropic Mirror Descent applied to $\Psif$ can be written as
\begin{align} \label{eq:EMDeta}
\mu_{n+1} (\rmd \theta) = \frac{ \mu_n (\rmd \theta) \exp \left[- \eta' (\bmuf[\mu_n](\theta)+\cte)\right]} {\mu_n(\exp\left[-\eta'(\bmuf[\mu_n]+\cte) \right])} \eqsp, \quad n\in\nset^\star \eqsp,
\end{align}
and we obtain in a similar fashion that an approximated version of this iterative scheme is
$$
\mu_{n+1}(\rmd \theta) = \mu_n (\rmd \theta)\left[1 - \eta' \left(\bmuf[\mu_n](\theta)- \mu_n(\bmuf[\mu_n]) \right) \right] \eqsp, \quad n\in\nset^\star \eqsp.
$$
Thus, for the two approximated formulas above to coincide, we need to set $\eta' =\eta \lrb{(\alpha-1)(\mu_n(\bmuf[\mu_n]) + \cte) +1}^{-1}$. Now coming back to \eqref{eq:EMDeta}, we see that this leads us to consider the update formula given by
\begin{align}\label{eq:renyiEMD}
\mu_{n+1}(\rmd \theta) = \frac{\mu_n(\rmd \theta)  \exp \lrb{- \eta \frac{\bmuf[\mu_n](\theta)}{(\alpha-1)(\mu_n(\bmuf[\mu_n]) + \cte) + 1} } }{\mu_n \lr{\exp \lrb{- \eta \frac{\bmuf[\mu_n]}{(\alpha-1)(\mu_n(\bmuf[\mu_n]) + \cte) + 1}}}} \eqsp, \quad n\in\nset^\star \eqsp.
\end{align}
Observe then that \eqref{eq:renyiEMD} can again be seen as an Entropic Mirror Descent, but applied this time to the objective function defined for all $\alpha \in \rset \setminus \lrcb{0,1}$ by
\begin{align*}
\PsifAR(\mu) &\eqdef \frac{1}{\alpha(\alpha-1)} \log \lr{\int_\Yset \mu k(y)^{\alpha} p(y)^{1-\alpha} \nu(\rmd y) + (\alpha-1) \cte} \eqsp,
\end{align*}
meaning we have applied the monotonic transformation
$$
u \mapsto  \frac{1}{\alpha(\alpha-1)} \log \lr{\alpha(\alpha-1)u + \alpha + (1-\alpha)\int_\Yset p(y) \nu(\rmd y) + (\alpha-1)\cte}
$$
to the initial objective function $\Psif$ (see \Cref{sec:renyiD} for the derivation of \eqref{eq:renyiEMD} based on the objective function $\PsifAR$). Hence, in the spirit of Renyi's $\alpha$-divergence gradient-based methods for Variational Inference (e.g \cite{2015arXiv151103243H, 2016arXiv160202311L}), we can motivate the iterative scheme \eqref{eq:renyiEMD} by observing that we recover the Variational Renyi bound introduced in \cite{2016arXiv160202311L} up to a constant $-\alpha^{-1}$ when we let $p = p(\cdot, \data)$, $\cte = 0$ and $\alpha > 0$ in $\PsifAR$. For this reason we call the algorithm given by \eqref{eq:renyiEMD} the \textit{Renyi Descent} thereafter.

Contrary to the Entropic Mirror Descent applied to $\Psif$, the Renyi Descent now shares the same first-order approximation as the Power Descent. This might explain why the behavior of the Entropic Mirror Descent applied to $\Psif$ and of the Power Descent differed greatly when $\alpha < 1$ in the numerical experiments from \cite{daudel2020infinitedimensional} despite their theoretical connection through the \aei descent framework (the former performing poorly numerically compared to the later as the dimension increased). 

Strikingly, we can prove an $O(1/N)$ convergence rate towards the global optimum for the Renyi Descent. Letting $\cte' \in \Rset$, denoting by $\Domain^{AR}$ an interval of $\Rset$ such that for all $\theta \in \Tset$ and all $\mu \in \meas{1}(\Tset)$,
$$
\frac{\bmuf(\theta) + 1/(\alpha-1)}{(\alpha-1) (\mu(\bmuf) + \cte) + 1} + \cte' \quad \mbox{and} \quad \frac{\mu(\bmuf) + 1/(\alpha-1)}{(\alpha-1) (\mu(\bmuf) + \cte) + 1} + \cte' \in \Domain^{AR}
$$
and introducing the assumption on $\eta$
\begin{hyp}{A}
\item\label{hyp:gamma} For all $v \in \Domain^{AR}$, $1 - \eta (\alpha -1)(v-\cte') \geq 0$.
\end{hyp}
we indeed have the following convergence result.

\begin{thm}
\label{thm:Renyi}
Assume \ref{hyp:positive} and \ref{hyp:gamma}. Let $\alpha \in \rset \setminus \lrcb{1}$ and let $\cte$ be such that $(\alpha -1) \cte > 0$. Define $\binfty \eqdef \sup_{\theta \in \Tset, \mu \in \meas{1}(\Tset)} |\bmuf(\theta) + 1/(\alpha-1)|$ and assume that $\binfty < \infty$. Moreover, let $\mu_1\in\meas{1}(\Tset)$ be such that
$\Psif(\mu_1)<\infty$. Then, the following assertions hold.
\begin{enumerate}[label=(\roman*)]
\item \label{item:admiss1} The sequence $(\mu_n)_{n\in\nstar}$ defined by \eqref{eq:renyiEMD} is well-defined and the sequence
$(\Psif (\mu_n))_{n\in\nstar}$ is non-increasing.

\item \label{item:admiss2} For all $N \in \nstar$, we have
\begin{align}\label{eq:rate}
   \Psif(\mu_N) - \Psif(\mu^\star) \leq \frac{\cteinf}{N} \left[ KL\couple[\mu^\star][\mu_1] + L\frac{ \ctesup}{\ctemono(\alpha-1)\cte} \Delta_1 \right]\eqsp,
\end{align}
where $\mu^\star$ is such that $\Psif(\mu^\star) = \inf_{\zeta \in \meas{1, \mu_1}(\Tset)} \Psif(\zeta)$, $\meas{1, \mu_1}(\Tset)$ denotes the set of probability measures dominated by $\mu_1$, $KL\couple[\mu^\star][\mu_1] = \int_\Tset \log \left(\rmd\mu^\star / \rmd\mu_1 \right) \rmd\mu^\star$, $\Delta_1 = \Psif(\mu_1) - \Psif(\mu^\star)$ and $\cteinf$, $L$, $\ctesup$, $\ctemono$ are finite constants defined in \eqref{eq:cteThm}.
\end{enumerate}
\end{thm}
The proof of this result is deferred to \Cref{sec:proofThmRenyi} and we present in the next example an application of this theorem to the particular case of mixture models.

\begin{ex} Let $\alpha \in \rset \setminus \lrcb{1}$, let $J \in \nstar$, let $\Theta = (\theta_1, \ldots, \theta_J) \in \Tset^J$, let $\mu_1 = J^{-1} \sum_{j=1}^{J} \delta_{\theta_j}$ and let $\Domain^{AR} = [- \frac{\binfty}{(\alpha-1)\cte} + \cte', \frac{\binfty}{(\alpha-1)\cte} + \cte']$ with $\cte' \in \Rset$. In addition, assume that $1-\eta |\cte|^{-1} \binfty > 0$. Then, taking $\cte' = - 3 \frac{\binfty}{(\alpha-1)\cte}$, we obtain
 $$
 \Psif(\mu_N) - \Psif(\mu^\star) \leq \frac{|\alpha-1|(\binfty+|\cte|)}{N} \lrb{\frac{\log J}{\eta} + \frac{\sqrt{2 \log(J)}\binfty}{(\alpha-1)\cte (1 - \eta |\cte|^{-1} \binfty)}} \eqsp,
 $$
where we have used that $KL\couple[\mu^\star][\mu_1] \leq \log J$, $\Delta_1 \leq \sqrt{2 \log J} \binfty$ and that the constants defined in \eqref{eq:cteThm} satisfy $\cteinf = \eta^{-1} |\alpha-1|(\binfty +|\cte|)$, $L = \eta^2 e^{\eta \frac{\binfty}{(\alpha-1)\cte} - \eta \cte'}$, $\ctesup = e^{\eta \frac{\binfty}{(\alpha-1)\cte} + \eta \cte'}$ and $\ctemono = (1-\eta |\cte|^{-1} \binfty)\eta e^{-\eta \frac{\binfty}{(\alpha-1)\cte} - \eta \cte'}$.
\end{ex}
To put things into perspective, notice that the Renyi Descent enjoys an $O(1/\sqrt{N})$ convergence rate as a Entropic Mirror Descent algorithm for the sequence $(\Psif(N^{-1} \sum_{n=1}^{N} \mu_n))_{N \in \nstar}$ under our assumptions when $\eta$ is proportional to $1/\sqrt{N}$, $N$ being fixed (see \cite{BECK2003167} or \cite[Theorem 4.2.]{MAL-050}).

The improvement thus lies in the fact that deriving an $O(1/N)$ convergence rate usually requires stronger smoothness assumptions on $\Psif$ \citep[Theorem 6.2]{MAL-050} that we do not assume in \Cref{thm:Renyi}. Furthermore, due to the monotonicity property, our result only involves the measure $\mu_N$ at time $N$ while typical Entropic Mirror Result are expressed in terms of the average $N^{-1} \sum_{n=1}^{N} \mu_n$. 

Finally, observe that the Renyi Descent becomes feasible in practice for any choice of kernel $K$ by letting $\mu$ be a weighted sum of Dirac measures i.e $\mu = \mulbd{}$ and by resorting to an unbiased estimate of $(\bmuf(\theta_j))_{1\leq j\leq J}$ (see Algorithm \ref{algo:mixture:RD} of \Cref{subsec:RDalgo}).

The theoretical results we have obtained are summarised in \Cref{table:recap} and we next move on to numerical experiments.
\begin{table}
  \caption{Summary of the theoretical results obtained in this paper compared to \cite{daudel2020infinitedimensional}}
  \label{table:recap}
  \centering
\begin{tabular}{llc}
      \toprule
    & Power Descent & Renyi Descent  \\
       \midrule
  \cite{daudel2020infinitedimensional} & $\alpha < 1$: convergence under restrictive assumptions;  & not covered\\
  & $\alpha >1$: $O(1/N)$ convergence rate &\\
      \midrule
  This paper &  $\alpha < 1$: full proof of convergence for mixture weights;  & $O(1/N)$  \\
  & extension to $\alpha = 1$ with $O(1/N)$ convergence rate & convergence rate\\
  \bottomrule
\end{tabular}
\end{table}

\section{Simulation study}
\label{sec:numerical}

Let the target $p$ be a mixture density of two $d$-dimensional Gaussian distributions multiplied by a positive constant $c$ such that $p(y) = c \times \left[ 0.5 \mathcal{N}(\boldsymbol{y}; -s \boldsymbol{u_d}, \boldsymbol{I_d}) + 0.5 \mathcal{N}(\boldsymbol{y}; s \boldsymbol{u_d}, \boldsymbol{I_d}) \right]$, where $\boldsymbol{u_d}$ is the $d$-dimensional vector whose coordinates are all equal to $1$, $s = 2$, $c = 2$ and $\boldsymbol{I_d}$ is the identity matrix. Given $J \in \nstar$, the approximating family is described by
$$
\lrc{y \mapsto \mu_{\lbd{}} k_h(y) = \sum_{j= 1}^{J} \lambda_j k_h(y- \theta_j) \eqsp : \eqsp \lbd{} \in \simplex_J, \theta_1, \ldots, \theta_J \in \Tset} \eqsp,
$$
where $K_{h}$ is a Gaussian transition kernel with bandwidth $h$ and $k_{h}$ denotes its associated kernel density.

Since the Power Descent and the Renyi Descent operate only on the mixture weights $\lbd{}$ of $\mu_{\lbd{}} k_h$ during the optimisation, a fully adaptive algorithm can be obtained by alternating $T$ times between an \textit{Exploitation step} where the mixture weights are optimised and an \textit{Exploration step} where the $\theta_1,  \ldots , \theta_J$ are updated, as written in Algorithm \ref{algo:adaptive}.

\SetInd{0.6em}{-1.8em}
\begin{algorithm}[H]
\caption{{\em Complete Exploitation-Exploration Algorithm}}
\label{algo:adaptive}
\textbf{Input}: $p$: measurable positive function, $\alpha $: $\alpha$-divergence parameter, $q_0$: initial sampler, $K_h$: Gaussian transition kernel, $T$: total number of iterations, $J$: dimension of the parameter set. \\
\textbf{Output}: Optimised weights $\lbd{}$ and parameter set $\Theta$. \\
Draw $\thetat[1][1], \ldots, \thetat[J][1]$ from $q_0$. \\
\For{$t = 1 \ldots T$}{
\begin{enumerate}[label={}]
\item \underline{Exploitation step} : \setlength{\parindent}{1cm} Set $\Theta = \{ \thetat[1][t], \ldots, \thetat[J][t] \}$. Perform the Power Descent or Renyi Descent and obtain the optimised mixture weights $\lbd{}$.
\item \underline{Exploration step} : \setlength{\parindent}{1cm} Perform any exploration step of our choice and obtain $\thetat[1][t+1], \ldots, \thetat[J][t+1]$.
\end{enumerate}
}
\end{algorithm}
Many choices of Exploration step can be envisioned in Algorithm \ref{algo:adaptive} since there is no constraint on $\lrcb{\theta_1, \ldots, \theta_J}$. Here, we consider the same Exploration step as the one they used in \cite{daudel2020infinitedimensional}: $h$ is set to be proportional to $J^{-1/(4 + d)}$ and the particles are updated by i.i.d sampling according to $\mulbd{} k_h$ (and we refer to \Cref{app:alter} for some details about alternative possible choices of Exploration step).

As for the Power Descent and Renyi Descent, we perform $N$ transitions of these algorithms at each time $t = 1 \ldots T$ according to Algorithm \ref{algo:mixture} and \ref{algo:mixture:RD}, in which the initial weights are set to be $[1/J, \ldots, 1/J]$, $\eta = \eta_0/\sqrt{N}$ with $\eta_0 > 0$ and $M$ samples are used in the estimation of $(\bmuf[\mulbd{}](\theta_{j,t}))_{1 \leq J}$ at each iteration $n = 1 \ldots N$. 
We take $J = 100$, $M \in \lrcb{100, 1000, 2000}$, $\alpha = 0.5$, $\cte = 0$, $\eta_0 = 0.3$ and the initial particles $\theta_1,  \ldots , \theta_J$ are sampled from a centered normal distribution $q_0$ with covariance matrix $5 \boldsymbol{I_d}$. We let $T=10$, $N=20$ and we replicate the experiment 100 times independently in dimension $d = 16$ for each algorithm. The convergence is assessed using a Monte Carlo estimate of the Variational Renyi bound introduced in \cite{2016arXiv160202311L} (which requires next to none additional computations).

The results for the Power Descent and the Renyi Descent are displayed on Figure \ref{fig:ComparePDandRD} below and we add the Entropic Mirror Descent applied to $\Psif$ as a reference.

\begin{figure}[ht!]
  \caption{Plotted is the average Variational Renyi bound for the Power Descent (PD), the Renyi Descent (RD) and the Entropic Mirror Descent applied to $\Psif$ (EMD) in dimension $d = 16$ computed over 100 replicates with $\eta_0 = 0.3$ and $\alpha = 0.5$ and an increasing number of samples $M$.}
  \label{fig:ComparePDandRD}
  \begin{tabular}{ccc}
  \includegraphics[width=4.3cm]{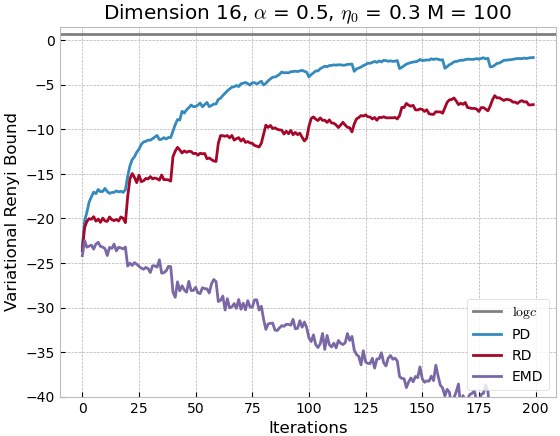} & \includegraphics[width=4.3cm]{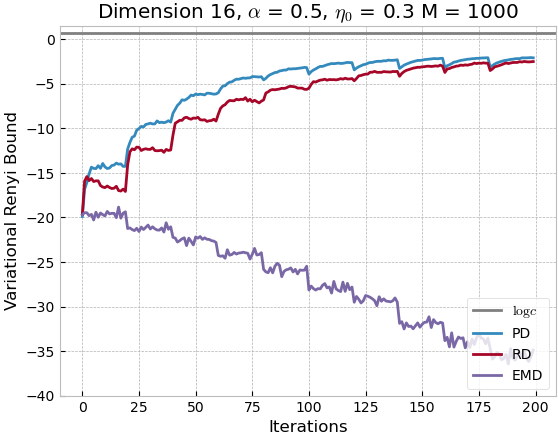} & \includegraphics[width=4.3cm]{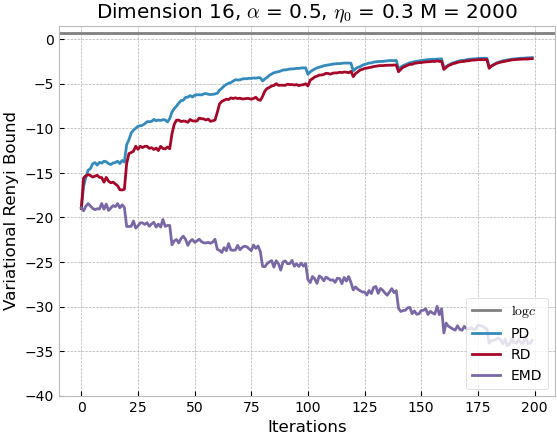}
  \end{tabular}
\end{figure}

We then observe that the Renyi Descent is indeed better-behaved compared to the Entropic Mirror Descent applied to $\Psif$, which fails in dimension $16$. Furthermore, it matches the performances of the Power Descent as $M$ increases in our numerical experiment, which illustrates the link between the two algorithms we have established in the previous section.

\paragraph{Discussion} From a theoretical standpoint, no convergence rate is yet available for the Power Descent algorithm when $\alpha <1$. An advantage of the novel Renyi Descent algorithm is then that while being close to the Power Descent, it also benefits from the Entropic Mirror Descent optimisation literature and as such $O(1/\sqrt{N})$ convergence rates hold, which we have been able to improve to $O(1/N)$ convergence rates.

A practical use of the Power Descent and of the Renyi Descent algorithms requires approximations to handle intractable integrals appearing in the update formulas so that the Power Descent applies the function $\Gamma(v) = [(\alpha - 1)v +1]^{\eta/(1-\alpha)}$ to an \textit{unbiased} estimator of the translated gradient $\bmuf(\theta) + \cte$ before renormalising, while the the Renyi Descent applies the Entropic Mirror Descent function $\Gamma(v) = e^{-\eta v}$ to a \textit{biased} estimator of $\bmuf[\mu_n](\theta)/(\mu_n(\bmuf[\mu_n]) + \cte + 1/(\alpha-1))$ before renormalising.


Finding which approach is most suitable between biased and unbiased $\alpha$-divergence minimisation is still an open issue in the literature, both theoretically and empirically \citep{biasedDomke2020, geffner2020difficulty, dhaka2021challenges}. Due to the exponentiation, considering the $\alpha$-divergence instead of Renyi's $\alpha$-divergence has for example been said to lead to high-variance gradients \citep{NIPS2017_6866, 2016arXiv160202311L} and low Signal-to-Noise ratio when $\alpha \neq 0$ \citep{geffner2020difficulty} during the stochastic gradient descent optimization.

In that regard, our work sheds light on additional links between unbiased and biased $\alpha$-divergence methods beyond the framework of stochastic gradient descent algorithms, as both the unbiased Power Descent and the biased Renyi Descent share the same first order approximation.

\section{Conclusion}
\label{sec:ccl}

We investigated algorithms that can be used to perform mixture weights optimisation for $\alpha$-divergence minimisation regardless of how the mixture parameters are obtained. We have established the full proof of the convergence of the Power Descent algorithm in the case $\alpha <1$ when we consider mixture models and bridged the gap with the case $\alpha =1$. We also introduced a closely-related algorithm called the Renyi Descent. We proved it enjoys an $O(1/N)$ convergence rate and illustrated in practice the proximity between these two algorithms when the number of samples $M$ increases.

Further work could include establishing theoretical results regarding the stochastic version of these two algorithms, as well as providing complementary empirical results comparing the performances of the unbiased $\alpha$-divergence-based Power Descent algorithm to those of the biased Renyi's $\alpha$-divergence-based Renyi Descent. 

%
%
%
%
%
%

\appendix

\newpage

\section{}

\subsection{Equivalence between \eqref{eq:optim:alpha} and \eqref{eq:GeneralProblem} with  $p(y) = p(y, \data)$}
\label{subsec:equi}

\begin{itemize}
  \item Case $\alpha = 1$ with $\falpha[1](u) = 1 - u + u \log(u)$ for all $u > 0$. Then,
  \begin{align*}
    \diverg[1] \couple[\mu K][\PP] &= \int_\Yset \falpha[1]\left(\frac{\mu k(y)}{p(y|\data)} \right) p(y|\data) \nu(\rmd y) \\
    &= \int_\Yset \mu k(y) \log \lr{\frac{\mu k(y)}{p(y|\data)}} \nu(\rmd y) + 0 \\
    & = \int_\Yset \mu k(y) \log \lr{\frac{\mu k(y)}{p(y, \data)}} \nu(\rmd y) + \log p(\data) \\
    &= \int_\Yset \falpha[1] \lr{\frac{\mu k(y)}{p(y,\data)}} p(y,\data) \nu(\rmd y) + 1 -p(\data) + \log p(\data)
  \end{align*}
  Thus,
  $$
  \arginf_{\mu \in \mathsf{M}} \diverg[1] \couple[\mu K][\PP] = \arginf_{\mu \in \mathsf{M}} \Psif[1](\mu; p) \quad \mbox{with} \quad p(y) = p(y, \data)
  $$
  \item Case $\alpha = 0$ with $\falpha[0](u) = u - 1 - \log(u)$ for all $u > 0$.
  \begin{align*}
   \diverg[0] \couple[\mu K][\PP] &= \int_\Yset \falpha[0]\left(\frac{\mu k(y)}{p(y|\data)} \right) p(y|\data) \nu(\rmd y) \\
   &= \int_\Yset - \log \left(\frac{\mu k(y)}{p(y|\data)} \right) p(y|\data) \nu(\rmd y) \\
   &= \int_\Yset - \log \left(\frac{\mu k(y)}{p(y, \data)} \right) p(y|\data) \nu(\rmd y) - \log p(\data) \\
   &= \frac{1}{p(\data)} \lrb{ \int_\Yset \falpha[1] \left(\frac{\mu k(y)}{p(y, \data)} \right) p(y,\data) \nu(\rmd y) + p(\data) - 1 - p(\data) \log p(\data)}
  \end{align*}
  Thus
  $$
  \arginf_{\mu \in \mathsf{M}} \diverg[0] \couple[\mu K][\PP] = \arginf_{\mu \in \mathsf{M}} \Psif[0](\mu; p) \quad \mbox{with} \quad p(y) = p(y, \data)
  $$
  \item Case $\alpha \in \rset \setminus \lrcb{1}$ with $\falpha(u) = \frac{1}{\alpha(\alpha-1)} \lrb{u^\alpha - 1 - \alpha(u-1)}$ for all $u > 0$.
  \begin{align}
  &\diverg \couple[\mu K][\PP] \nonumber\\
  &= \int_\Yset \falpha\left(\frac{\mu k(y)}{p(y|\data)} \right) p(y|\data) \nu(\rmd y) \nonumber \\
  &= \int_\Yset \frac{1}{\alpha(\alpha-1)} \lrb{\left(\frac{\mu k(y)}{p(y|\data)} \right)^{\alpha}- 1} p(y|\data) \nu(\rmd y) \nonumber \\
  &= p(\data)^{\alpha-1} \int_\Yset \frac{1}{\alpha(\alpha-1)} \lrb{\left(\frac{\mu k(y)}{p(y,\data)} \right)^{\alpha}- 1} p(y,\data) \nu(\rmd y) + \frac{p(\data)^{\alpha} - 1}{\alpha(\alpha-1)} \nonumber \\
  &= p(\data)^{\alpha-1} \int_\Yset \falpha \left(\frac{\mu k(y)}{p(y,\data)} \right) p(y,\data) \nu(\rmd y) + \frac{\alpha p(\data)^{\alpha-1} + (1-\alpha)p(\data)^{\alpha} - 1}{\alpha(\alpha-1) } \label{eq:limitOptimEqui}
  \end{align}
  Thus,
  $$
  \arginf_{\mu \in \mathsf{M}} \diverg \couple[\mu K][\PP] = \arginf_{\mu \in \mathsf{M}} \Psif(\mu; p) \quad \mbox{with} \quad p(y) = p(y, \data)
  $$
\end{itemize}

\subsection{\cite[Theorem 1]{daudel2020infinitedimensional} with $\GammaAlpha(v) = [(\alpha-1)v +1]^{\eta/(1-\alpha)}$}
\label{subsec:lem:admiss}

\begin{thm}[{\cite[Theorem 1]{daudel2020infinitedimensional} with $\GammaAlpha(v) = [(\alpha-1)v +1]^{\eta/(1-\alpha)}$}]
\label{thm:admiss}
Assume that $p$ and $k$ are as in \ref{hyp:positive}. Let $\alpha \in \rset \setminus \lrcb{1}$, let $\cte$ be such that $(\alpha-1)\cte \geq 0$, let $\mu \in \meas{1}(\Tset)$ and let $\eta \in (0,1]$ be such that
\begin{align}\label{eq:admiss}
0 <\mu(\GammaAlpha(\bmuf + \cte))<\infty \eqsp
\end{align}
holds and $\Psif(\mu) < \infty$. Then, the two following assertions hold.
\begin{enumerate}[label=(\roman*)]
\item \label{item:mono1Prev} We have  $\Psif \circ \iteration (\mu) \leq \Psif(\mu)$.
\item \label{item:mono2Prev} We have $\Psif \circ \iteration (\mu) =\Psif(\mu)$ if and only if $\mu=\iteration (\mu)$.
\end{enumerate}
\end{thm}

\section{}

\subsection{Proof that \ref{hyp:compactFull} is satisfied in \Cref{ex:thm:limit}}
\label{sec:exGaussian}

\begin{proof}[Proof that \ref{hyp:compactFull} is satisfied in \Cref{ex:thm:limit}] \ \\

We have $k_h(\theta,y) = \frac{e^{-\| y - \theta \|^2 / (2h^2)}}{(2 \pi h^2)^{d/2}}$ and $p(y) = c \times \lrb{ 0.5 \frac{e^{-\| y - \theta^\star_1 \|^2 / 2}}{(2 \pi)^{d/2}}+ 0.5 \frac{e^{-\| y - \theta^\star_2 \|^2 / 2}}{(2 \pi)^{d/2}}}$ for all $\theta \in \Tset$ and all $y \in \Yset$. Recall that by assumption $\Tset = \mathcal{B}(0,r) \subset \rset^d$ with $r >0$. Then, for all $\alpha \in [0,1)$, we are interested in proving
\begin{align}\label{eq:BmuBdd}
\int_\Yset {\sup_{\theta \in \Tset}k(\theta,y)} \times \sup_{\theta' \in \Tset} \lr{\frac{k(\theta',y)}{p(y)}}^{\alpha-1} \nu(\rmd y)<\infty
\end{align}
and
\begin{align}\label{eq:PsiMuOne}
\int_\Yset \sup_{\theta \in \Tset} \left| \log\lr{\frac{k_h(\theta,y)}{p(y)}} \right| p(y) \nu(\rmd y) < \infty \eqsp.
\end{align}

\begin{enumerateList}
\item We start by proving \eqref{eq:BmuBdd}. First note that for all $\theta, \theta' \in \Tset$ and for all $y \in \Yset$ we can write
\begin{align*}
   \frac{k_h(\theta, y)}{k_h(\theta ', y)} &= e^{\frac{-\| y - \theta \|^2 + \| y - \theta ' \|^2}{2h^2}} = e^{\frac{2 <y, \theta - \theta '> - \| \theta \|^2 + \| \theta '\|^2}{2h^2}} \\
  & \leq e^{\frac{2 |<y,\theta - \theta'>| + \| \theta\|^2 + \|\theta '\|^2 }{2h^2}} \leq  e^{\frac{\|y \| \| \theta - \theta'  \| + r^2}{h^2}} \eqsp.
\end{align*}
from which we deduce that for all $\theta, \theta' \in \Tset$ and for all $y \in \Yset$,
\begin{align}\label{eq:boundRatioGaussian}
 \frac{k_h(\theta, y)}{k_h(\theta ', y)}  \leq e^{\frac{\|y \| 2 r + r^2}{h^2}} \eqsp
  \end{align}
and that
$$
\int_\Yset {\sup_{\theta \in \Tset}k(\theta,y)} \times \sup_{\theta' \in \Tset} \lr{\frac{k(\theta',y)}{p(y)}}^{\alpha-1} \nu(\rmd y) \leq \int_\Yset k(\theta,y) e^{\frac{\|y \| 2 r + r^2}{h^2}} \sup_{\theta' \in \Tset}  \lr{\frac{k(\theta',y)}{p(y)}}^{\alpha-1} \nu(\rmd y).
$$
Additionally, Jensen's inequality applied to the concave function $u \mapsto u^{1-\alpha}$ implies
\begin{align*}
\int_\Yset k(\theta,y) e^{\frac{\|y \| 2 r + r^2}{h^2}} \sup_{\theta' \in \Tset}  \lr{\frac{k(\theta',y)}{p(y)}}^{\alpha-1} \nu(\rmd y) & \leq \lr{ \int_\Yset k(\theta,y) e^{\frac{\|y \| 2 r + r^2}{(1-\alpha)h^2}} \sup_{\theta' \in \Tset}  \frac{p(y)}{k(\theta',y)} \nu(\rmd y)}^{1-\alpha} \\
& \leq \lr{ \int_\Yset \sup_{\theta, \theta' \in \Tset} \frac{k_h(\theta,y)}{k_h(\theta ', y)}  e^{\frac{\|y \| 2 r + r^2}{(1-\alpha)h^2}}  p(y) \nu(\rmd y)}^{1-\alpha}
\end{align*}
Now using \eqref{eq:boundRatioGaussian}, we can deduce
$$
\int_\Yset \sup_{\theta, \theta' \in \Tset} \frac{k_h(\theta,y)}{k_h(\theta ', y)}  e^{\frac{\|y \| 2 r + r^2}{(1-\alpha)h^2}} p(y) \nu(\rmd y) \leq \int_\Yset e^{\frac{\|y \| 2 r + r^2}{h^2}(1 + \frac{1}{1-\alpha})} p(y) \nu(\rmd y) < \infty \eqsp,
$$
which yields the desired result.
\item We now prove \eqref{eq:PsiMuOne}. For all $y \in \Yset$ and all $\theta \in \Tset$, we have
\begin{align*}
& e^{-\sup_{\theta \in \Tset} \frac{\| y - \theta \|^2}{2h^2}}\leq (2 \pi h^2)^{d/2} k_h(\theta, y)  \leq 1 \\
& e^{-\max_{i \in \lrcb{1,2}} \frac{\| y - \theta_i^\star \|^2}{2}}\leq  c^{-1} (2 \pi )^{d/2} p(y) \leq 1
\end{align*}
and we can deduce for all $y \in \Yset$ and  all $\theta \in \Tset$
\begin{align}\label{eq:BoundLogGaussian}
\left| \log\lr{ \frac{k_h(\theta, y)}{p(y)}} \right| & \leq \sup_{\theta \in \Tset} \frac{\| y - \theta \|^2}{2h^2}  + \max_{i \in \lrcb{1,2}} \frac{\| y - \theta_i^\star \|^2}{2} + d |\log h| + |\log c| \nonumber \\
& \leq \frac{(\| y \| + r)^2}{2}\lrb{\frac{1}{h^2} + 1} + d |\log h| + |\log c| \eqsp.
\end{align}
Since we have
$$
\int_\Yset \lr{ \frac{(\| y \| + r)^2}{2}\lrb{\frac{1}{h^2} + 1} + d |\log h| + |\log c| } p(y) \nu(\rmd y) < \infty
$$
we deduce that \eqref{eq:PsiMuOne} holds.
\end{enumerateList}
\end{proof}

%


\subsection{Proof of \Cref{thm:limit}}
\label{subsec:prooflimit}

We start with some preliminary results. Let $\zeta, \zeta' \in \meas{1}(\Tset)$. Recall that we say that $\zeta \mathcal{R} \zeta'$ if and only if $\zeta K=\zeta' K$ and that $\meas{1, \zeta}(\Tset)$ denotes the set of probability measures dominated by $\zeta$.

\begin{lem}
\label{lem:fixed:repulsive:prelim:R} Assume \ref{hyp:positive}. Let $\mathsf M$ be a convex subset of $\meas{1}(\Tset)$ and let $\zeta_1, \zeta_2 \in \meas{1}(\Tset)$ be such that
$$
\Psif(\zeta_1) = \Psif(\zeta_2) = \inf_{\zeta \in \mathsf M} \Psif(\zeta).
$$
Then, we have $\zeta_1 \mathcal{R} \zeta_2$.
\end{lem}
\begin{proof}
 For all $y \in \Yset$, set $u_y= \zeta_1 k(y)/ p(y)$ and $v_y= \zeta_2 k(y) /p(y)$. Then, for all $y \in \Yset$ and for all $t \in (0,1)$, $\falpha(t u_y + (1- t) v_y) \leq t \falpha(u_y) + (1-t) \falpha(v_y)$ by convexity of $\falpha$ and we obtain
\begin{align}\label{eq:convex:equality1}
\Psif(t \zeta_1 + (1-t) \zeta_2) \leq t \Psif(\zeta_1) + (1-t) \Psif(\zeta_2) = \inf_{\zeta \in \mathsf M} \Psif(\zeta) \eqsp.
\end{align}
Furthermore, $t \zeta_1 + (1-t) \zeta_2 \in \mathsf M$ which implies that we have equality in \eqref{eq:convex:equality1}.

Consequently, for all $t \in (0,1)$ :
$$
\int_\Yset \underbrace{\left[  t \falpha(u_y) + (1-t) \falpha(v_y) - \falpha(t u_y + (1- t) v_y)\right] }_{\geq 0}  p(y) \nu(\rmd y) = 0 \eqsp.
$$
Now using that $\falpha$ is strictly convex, we deduce that for $p$-almost all $y \in \Yset$, $\zeta_1 k(y) = \zeta_2 k(y)$ that is $\zeta_1 \mathcal{R} \zeta$.
\end{proof}

\begin{lem}\label{lem:fixed:point:inf}
Assume \ref{hyp:positive}. Let $\alpha \in \rset \setminus \lrcb{1}$, let $\cte$ be such that $(\alpha-1) \cte \geq 0$ and let $\muf \in \meas{1}(\Tset)$ be a fixed point of $\iteration$. Then,
\begin{equation}\label{eq:inf:fixed}
\Psif(\muf) = \inf_{\zeta \in \meas{1,\muf}(\Tset)} \Psif(\zeta)\eqsp.
\end{equation}
Furthermore, for all $\zeta \in \meas{1,\muf}(\Tset)$, $\Psif(\muf) = \Psif(\zeta)$ implies that $\muf \mathcal{R} \zeta$.

\end{lem}

\begin{proof}
Let $\zeta\in \meas{1, \muf}(\Tset)$ be such that $\Psif(\zeta) \leq \Psif(\muf)$. We have that
\begin{align}\label{eq:infOne}
\zeta\left( \bmuf[\muf] - \muf(\bmuf[\muf]) \right) \leq \Psif(\zeta) - \Psif(\muf) \leq 0 \eqsp.
\end{align}
Furthermore, since $\muf$ is a fixed point of $\iteration$, $\GammaAlpha(\bmuf[\muf] + \cte)$, hence $|\bmuf[\muf] + \cte + 1/(\alpha-1)|$ is $\muf$-almost all constant. In addition, $\bmuf[\muf] + \cte + 1/(\alpha-1)$ is of constant sign by assumption on $\cte$. Since $\zeta\preceq \muf$, we thus deduce that
\begin{align*}
\zeta\left( \bmuf[\muf] - \muf(\bmuf[\muf]) \right) = 0 \eqsp.
\end{align*}
Combining this result with \eqref{eq:infOne} yields $\Psif(\zeta) = \Psif(\muf)$ and we recover \eqref{eq:inf:fixed}.

Finally, assume there exists $\zeta \in \meas{1, \muf}(\Tset)$ such that $\Psif(\muf) = \Psif(\zeta)$. Then, since $\meas{1, \muf}(\Tset)$ is a convex set, we have by \Cref{lem:fixed:repulsive:prelim:R} that $\muf \mathcal{R} \zeta$.
\end{proof}

We now move on to the proof of \Cref{thm:limit}.

\begin{proof}[Proof of \Cref{thm:limit}] For convenience, we define the notation $\Psif[\alpha, \Theta] (\lbd{}) \eqdef \Psif \left( \mulbd  \right)$ for all $\lbd{} \in \simplex_J$. In this proof, we will use the equivalence relation $\mathcal{R}$ defined by: $\zeta \mathcal{R} \zeta'$ if and only if $\zeta K=\zeta' K$ and we write $\meas{1, \zeta}(\Tset)$ the set of probability measures dominated by $\zeta$.

\begin{enumerateList}

\item \label{eq:thm:one} \textit{Any possible limit of convergent subsequence of $(\lbd{n})_{n\in\nset^\star}$ is a fixed point of $\iteration^{\mathrm{mixt}}$. }

First note that by \ref{hyp:compact}, we have that $|\Psif[\alpha, \Theta](\lbd{})| < \infty$ and that \eqref{eq:admiss} is satisfied for all $\mulbd$ such that $\lbd{} \in \simplex_J$. This means that the sequence $(\lbd{n})_{n\in\nset^\star}$ defined by \eqref{eq:iteration:mixture} is well-defined, that the sequence $(\Psif[\alpha, \Theta](\lbd{n}))_{n\in\nset^\star}$ is lower-bounded and that $\Psif[\alpha, \Theta](\lbd{n})$ is finite for all $n\in\nset^\star$. As $(\Psif[\alpha, \Theta](\lbd{n}))_{n\in\nset^\star}$ is nonincreasing by \Cref{thm:admiss}-\ref{item:mono1Prev}, it converges in $\mathbb{R}$ and in particular we have
$$
  \lim \limits_{n\to \infty} \Psif[\alpha, \Theta] \circ \iteration^{\mathrm{mixt}}(\lbd{n}) - \Psif[\alpha, \Theta](\lbd{n}) = 0 \eqsp.
$$
Let $(\lbd{\varphi(n)} )_{n\in\nset^\star}$ be a convergent subsequence of $(\lbd{n})_{n\in\nset^\star}$ and denote by $\bar{\lbd{}}$ its limit. Since the function $\lbd{} \mapsto \Psif[\alpha, \Theta] \circ \iteration^{\mathrm{mixt}}(\lbd{}) - \Psif[\alpha, \Theta] (\lbd{})$ is continuous we obtain that $\Psif[\alpha, \Theta] \circ \iteration^{\mathrm{mixt}}(\bar{\lbd{}}) =\Psif[\alpha, \Theta] (\bar{\lbd{}})$ and hence by \Cref{thm:admiss}-\ref{item:mono2Prev}, $\bar{\lbd{}}$ is a fixed point of $\iteration^{\mathrm{mixt}}$.

\item \label{eq:thm:two}  \textit{The set  $F = \lrc{\lbd{} \in \simplex_J \eqsp : \eqsp \lbd{} = \iteration^{\mathrm{mixt}}(\lbd{})}$ of fixed points of  $\iteration^{\mathrm{mixt}}$ is finite}.

For any subset $R\subset \lrcb{1,\ldots,J}$, define
\begin{align*}
  \simplex_{J,R}&=\lrc{ \lbd{} \in \simplex_J \eqsp : \eqsp \forall i \in R^c, \lambda_i = 0,\forall j \in R^c, \lambda_j \neq 0}\eqsp,\\
  \tilde \simplex_{J,R}&=\lrc{ \lbd{} \in \simplex_J \eqsp : \eqsp \forall i \in R^c, \lambda_i = 0}\eqsp,
\end{align*}
and write $$F=\bigcup \limits_{R \subset \{ 1,  \ldots , J \}} (S_{J,R} \cap F)\eqsp.$$
In order to show that $F$ is finite, we prove by contradiction that for any $R \subset \{ 1,  \ldots , J \}$, $S_{J,R} \cap F$ contains at most one element. Assume indeed the existence of two distinct elements $\lbd{}\neq  \lbdp{}$ belonging to $ S_{J,R} \cap F$. Since $\meas{1,\mulbd[\lbd{}]}(\Tset)=\meas{1,\mulbd[\lbdp{}]}(\Tset)=\lrcb{\mulbd[\lbdpp{}]\eqsp: \lbdpp{} \in \tilde \simplex_{J,R}}$, \Cref{lem:fixed:point:inf} implies that
\begin{equation*}
\Psif[\alpha,\Theta](\lbd{}) = \inf_{\lbdpp{} \in \tilde \simplex_{J,R} } \Psif[\alpha,\Theta]\lr{\lbdpp{}}=\Psif[\alpha,\Theta](\lbdp{})  \eqsp.
\end{equation*}
Applying again \Cref{lem:fixed:point:inf}, we get $\mulbd[\lbd{}] \mathcal{R} \mulbd[\lbdp{}]$, that is, $\mulbd K=\mulbd[\lbdp{}]K$. This means that $\sum_{j=1}^J (\lambda_j-\lambda'_j)K(\theta_j,\cdot)$ is the null measure, which in turns implies the identity $\lbd{}=\lbdp{}$ since the family of measures $\lrcb{K(\theta_1,\cdot),\ldots,K(\theta_J,\cdot)}$ is assumed to be linearly independent.

\item \textit{Conclusion.}

According to \Cref{lem:fixed:repulsive:prelim:R} applied to the convex subset of measures $\mathsf{M}=\simplex_{J}$, the function $\Psif[\alpha,\Theta]$ attains its global infimum at a unique $\lbd{\star}\in \simplex_{J}$. The uniqueness of $\lbd{\star}$ actually follows from the fact that, as shown above, $\mulbd[\lbd{}] \mathcal{R} \mulbd[\lbdp{}]$ if and only if $\lbd{}=\lbdp{}$.
Then, by \Cref{thm:admiss}-\ref{item:mono1Prev} and by definition of $\lbd{\star}$
$$
\Psif[\alpha, \Theta] \circ \iteration^{\mathrm{mixt}}({\lbd{\star}}) \leq \Psif[\alpha, \Theta]({\lbd{\star}}) = \inf_{\lbd{}' \in \simplex_J} \Psif[\alpha, \Theta](\lbd{}')  \leq \Psif[\alpha, \Theta] \circ \iteration^{\mathrm{mixt}}({\lbd{\star}})\eqsp,
$$
and hence, $\Psif[\alpha, \Theta] \circ \iteration^{\mathrm{mixt}}({\lbd{\star}}) =\Psif[\alpha, \Theta]({\lbd{\star}})$, showing that $\lbd{\star} \in F$ by \Cref{thm:admiss}-\ref{item:mono2Prev}. Since by \ref{eq:thm:two}, $F$ is finite, there exists $L\geq 1$ such that $F=\lrcb{ \lbd{}^\ell \eqsp: 1 \leq \ell \leq L} $, where for $i\neq j$, $\lbd{}^i\neq \lbd{}^j$. Without any loss of generality, we set $\lbd{}^1=\lbd{\star}$ to simplify the notation.

We now introduce a sequence $(W_{\ell})_{1 \leq {\ell} \leq L}$ of disjoint open neighborhoods of $(\lbd{}^\ell)_{1\leq \ell \leq  L}$ such that for any $\ell \in \{1,\ldots,L\}$,
\begin{equation} \label{eq:nojump}
  \iteration^{\mathrm{mixt}}(W_{\ell}) \cap \lr{\bigcup_{j \neq \ell} W_{j}}=\emptyset
\end{equation}
This is possible since $\iteration^{\mathrm{mixt}}(\lbd{}^{\ell}) = \lbd{}^{\ell}$ and $\lbd{} \mapsto \iteration^{\mathrm{mixt}}(\lbd{})$ is continuous.

By \ref{eq:thm:one} , the set $F$ contains all the possible limits of any subsequence of $(\lbd{n})_{n\in\nset^\star}$. As a consequence, there exists $N > 0$ such that for all $n \geq N$, $\lbd{n} \in \bigcup_{1\leq {\ell}\leq L} W_{\ell}$. Combining with \eqref{eq:nojump}, there exists $\ell \in \{1, \ldots ,L \}$ such that for all $n \geq N$, $\lbd{n} \in W_{\ell}$. Therefore $\lbd{}^\ell$ is the only possible limit of any convergent subsequence of $(\lbd{n})_{n\in\nset^\star}$ and as a consequence, $\lim_{n\to\infty}\lbd{n}=\lbd{}^\ell$.

Thus, the sequence $(\mu_{\lbd{n},\Theta})_{n\in\nset^\star}$ weakly converges to $\mu_{\lbd{}^\ell,\Theta}$ as $n \to \infty$ and \Cref{thm:repulsive} can be applied. Since $\lbd{1} \in \simplex_J^+$, we have $\meas{1, \mu_{\lbd{1},\Theta}}(\Tset)=\lrcb{\mu_{\lbdp{},\Theta}\eqsp:\lbdp{} \in \simplex_J}$ and \Cref{thm:repulsive}-\ref{item:rep2} then shows that $\mu_{\lbd{}^\ell,\Theta}$ is the global arginf of $\Psif$ over all $\lrcb{\mu_{\lbdp{},\Theta}\eqsp:\lbdp{} \in \simplex_J}$. Therefore, $\ell=1$, i.e., $\lbd{}^\ell=\lbd{}^1=\lbd{\star}$ and
$$
\Psif[\alpha,\Theta](\lbd{\star})= \inf_{\lbd{}' \in \simplex_J} \Psif[\alpha, \Theta](\lbd{}') \eqsp.
$$
\end{enumerateList}
\end{proof}

\subsection{The Power Descent for mixture models: practical version}
\label{subsec:aeiprac}

The algorithm below provides one possible approximated version of the Power Descent algorithm, where we have set $\GammaAlpha(v) = [(\alpha-1)v+1]^{\frac{\eta}{1-\alpha}}$ with $\eta \in (0,1]$.

\SetInd{0.6em}{-1.8em}
\begin{algorithm}[!h]
\caption{{\em Practical version of the Power Descent for mixture models}}
\label{algo:mixture}
\textbf{Input:} $p$: measurable positive function, $K$: Markov transition kernel, $M$: number of samples, $\Theta = \{\theta_1, \ldots, \theta_J\} \subset \Tset$: parameter set, $\GammaAlpha(v) = [(\alpha-1)v+1]^{\frac{\eta}{1-\alpha}}$ with $\eta \in (0,1]$, $N$: total number of iterations. \\
\textbf{Output:} Optimised weights $\lbd{}$. \\

Set $\lbd{} = [\lambda_{1, 1}, \ldots, \lambda_{J,1}]$.\\
\For{$n = 1 \ldots N$}{
\begin{enumerate}[label={}]
\item \underline{Sampling step} : \setlength{\parindent}{1cm} Draw independently $M$ samples $Y_{1}, \ldots, Y_{M}$ from $\mulbd k$.
\item \underline{Expectation step} : \setlength{\parindent}{1cm} Compute $\boldsymbol{B}_{\lbd{}} = (b_{j})_{1 \leq j \leq J}$ where for all $j = 1 \ldots J$ 
\begin{align*}
b_{j} = \dfrac{1}{M} \sum \limits_{m=1}^M \frac{k(\theta_j, Y_{m})}{\mulbd k(Y_{m})}  \falpha'\left( \frac{\mulbd k(Y_{m})}{p(Y_{m})} \right)
\end{align*} 

\noindent and deduce $\boldsymbol{W}_{\lbd{}}  = (\lambda_j \GammaAlpha(b_{j} + \cte))_{1\leq j\leq J}$ and $w_{\lbd{}}  = \sum_{j=1}^J \lambda_j \GammaAlpha(b_{j} + \cte)$.
\item \underline{Iteration step} : \setlength{\parindent}{1cm} Set
\begin{align*}
\lbd{} \leftarrow \frac{1}{w_{\lbd{}} } \boldsymbol{W}_{\lbd{}}
\end{align*}
\end{enumerate}
}
\end{algorithm}

\section{}
\subsection{Proof of \Cref{lem:extenstionAlpha1}}
\label{sec:FirstLemma}

We first state \ref{hypLimAlpha1}, which summarises the necessary convergence and differentiability assumptions needed in the proof of \cref{lem:extenstionAlpha1}.

\begin{hypD}{D}
\item
\begin{enumerate}[label=(\roman*)]
\item  we have $\mathlarger\int_\Yset \sup \limits_{\theta \in \Tset} k(\theta ,y) \times \sup \limits_{\theta ' \in \Tset } \lr{ \frac{k(\theta ',y)}{p(y)}}^{\alpha-1}  \nu(\rmd y)<\infty$; \label{hypD:1}
\item we have $\mathlarger\int_\Yset \sup \limits_{\theta \in \Tset} k(\theta ,y) \times \sup \limits_{\theta ' \in \Tset } \left| \log \left(\frac{k(\theta' ,y)}{p(y)} \right)  \right| \times \sup \limits_{\theta '' \in \Tset } \lr{ \frac{k(\theta '',y)}{p(y)}}^{\alpha-1} \nu(\rmd y)<\infty$; \label{hypD:2}
\item we have $\mathlarger\int_\Yset \inf \limits_{\theta \in \Tset} k(\theta ,y) \times \inf \limits_{\theta ' \in \Tset } \lr{ \frac{k(\theta ',y)}{p(y)}}^{\alpha-1}  \nu(\rmd y) > 0$. \label{hypD:3}
\end{enumerate} \label{hypLimAlpha1}
\end{hypD}

Note that these assumptions are mild if we assume that $\Tset$ is a compact metric space, which is generally the case. Assumption \ref{hypLimAlpha1}-\ref{hypD:3} is only required when $\alpha > 1$ to ensure that the quantity $[(\alpha-1)(\bmuf+\cte) +1]^{\frac{\eta}{1-\alpha}}$ is bounded from above. This assumption could also be replaced by the assumption that $\cte$ is such that $(\alpha-1) \cte > 0$.

\begin{proof}[Proof of \cref{lem:extenstionAlpha1}]
For all $\theta \in \Tset$, the Dominated Convergence Theorem and \ref{hypLimAlpha1}-\ref{hypD:1} yield
$$
\lim_{\alpha \to 1} (\alpha-1) (\bmuf(\theta)+\kappa) +1 = \lim_{\alpha \to 1} \int_\Yset k(\theta,y) \left(\frac{\mu k(y)}{p(y)}\right)^{\alpha-1} \nu(\rmd y)+0 = 1 \eqsp.
$$
Then, using \ref{hypLimAlpha1}-\ref{hypD:2} we have that for all $\theta \in \Tset$,
\begin{align*}
\lim_{\alpha \to 1} \left[(\alpha-1)(\bmuf(\theta) + \cte) +1 \right]^{\frac{\eta}{1-\alpha}}&= \exp\left( \lim_{\alpha \to 1} - \eta \frac{\log\left[(\alpha-1)(\bmuf(\theta) + \cte) +1\right]}{\alpha-1}  \right) \\
&= \exp \left(\lim_{\alpha \to 1} - \eta \frac{\int_\Yset k(\theta,y)  \left( \frac{\mu k(y)}{p(y)} \right)^{\alpha-1} \log \left(\frac{\mu k(y)}{p(y)} \right)\nu(\rmd y) + \cte}{\int_\Yset k(\theta,y) \left( \frac{\mu k(y)}{p(y)} \right)^{\alpha-1} \nu(\rmd y) + (\alpha-1)\cte}  \right) \\
& = \exp \left[- \eta \int_\Yset k(\theta,y) \log \left(\frac{\mu k(y)}{p(y)} \right) \nu(\rmd y) \right] \exp \lr{-\eta \cte}
\end{align*}
In addition, by the Dominated Convergence Theorem (and  \ref{hypLimAlpha1}-\ref{hypD:3} when $\alpha > 1$), we have
$$
\lim_{\alpha \to 1} \mu \left([(\alpha-1)(\bmuf+\cte)+1]^{\frac{\eta}{1-\alpha}} \right) = \mu\left(  \exp \left[- \eta \int_\Yset k(\cdot,y) \log \left(\frac{\mu k(y)}{p(y)} \right) \nu(\rmd y)  \right] \right) \exp \lr{-\eta \cte} \eqsp.
$$
Thus,
$$
\lim_{\alpha\to 1} [\iteration (\mu)](h) = \int_\Tset \frac{\mu(\rmd \theta) h(\theta) e^{- \eta \int_\Yset k(\theta,y) \log \left(\frac{\mu k(y)}{p(y)} \right) \nu(\rmd y)}}{ \mu\left(  e^{- \eta \int_\Yset k(\cdot,y) \log \left(\frac{\mu k(y)}{p(y)} \right) \nu(\rmd y)} \right) } = [\iteration[1](\mu)](h) \eqsp.
$$
\end{proof}

\subsection{Derivation of the update formula for the Renyi Descent}
\label{sec:renyiD}

For all $\alpha \in \rset \setminus \lrcb{0, 1}$ and $\cte$ such that $(\alpha-1)\cte \geq 0$, we are interested applying the Entropic Mirror Descent algorithm to the following objective function
\begin{align*}
\PsifAR(\mu) &\eqdef \frac{1}{\alpha(\alpha-1)} \log \lr{\int_\Yset \mu k(y)^{\alpha} p(y)^{1-\alpha} \nu(\rmd y) + (\alpha-1) \cte}
\end{align*}

\begin{lem} Assume \ref{hyp:positive}. The gradient of $\PsifAR(\mu)$ is given by $\theta \mapsto \frac{\bmuf[\mu](\theta) + 1/(\alpha-1)}{(\alpha-1)(\mu(\bmuf[\mu]) + \cte) + 1}$.
\end{lem}

\begin{proof} Let $\varepsilon > 0$ be small and let $\mu, \mu' \in \meas{1}(\Tset)$. Then,
\begin{align*}
\PsifAR(\mu + \varepsilon \mu') &= \frac{1}{\alpha(\alpha-1)} \log \lr{\int_\Yset [(\mu + \varepsilon \mu') k(y)]^{\alpha} p(y)^{1-\alpha} \nu(\rmd y) + (\alpha-1) \cte} \\
&= \frac{1}{\alpha(\alpha-1)} \log \lr{\int_\Yset \mu k(y)^\alpha \left[1 + \alpha \varepsilon \frac{\mu' k(y)}{\mu k(y)}\right] p(y)^{1-\alpha} \nu(\rmd y) + (\alpha-1) \cte + o(\varepsilon)}
\end{align*}
where we used that $(1+u)^{\alpha} = 1 +\alpha u + o(u)$ as $u \to 0$. Thus,
\begin{align*}
\PsifAR(\mu + \varepsilon \mu') & = \PsifAR(\mu) + \frac{1}{\alpha(\alpha-1)} \log \lr{ 1 + \alpha \varepsilon \frac{\int_\Yset \mu' k(y) \left( \frac{\mu k(y)}{p(y)} \right)^{\alpha-1} \nu(\rmd y) }{\int_\Yset \mu k(y)^\alpha p(y)^{1-\alpha} \nu(\rmd y) + (\alpha-1) \cte} + o(\varepsilon)} \\
&= \PsifAR(\mu) + \varepsilon \frac{1}{\alpha-1} \frac{\int_\Yset \mu' k(y) \left( \frac{\mu k(y)}{p(y)} \right)^{\alpha-1} \nu(\rmd y)}{\int_\Yset \mu k(y)^\alpha p(y)^{1-\alpha} \nu(\rmd y) + (\alpha-1) \cte} + o(\varepsilon) \\
&= \PsifAR(\mu) + \varepsilon \int_\Tset \mu'(\rmd \theta) \frac{1}{\alpha-1} \frac{\bmuf[\mu](\theta) + 1/(\alpha-1)}{\mu(\bmuf[\mu]) + \cte + 1/(\alpha-1)} + o(\varepsilon)
\end{align*}
using that $\log(1+u) = u + o(u)$ as $u \to 0$.
\end{proof}
Consequently, the iterative update formula for the Entropic Mirror Descent applied to the objective function $\PsifAR$ is given by
\begin{align*}
\mu_{n+1}(\rmd \theta) = \mu_n(\rmd \theta) \frac{e^{- \frac{\eta}{\alpha-1} \frac{\bmuf[\mu_n](\theta)}{\mu_n(\bmuf[\mu_n]) + \cte  + 1/(\alpha-1)}} }{\mu_n(e^{- \frac{\eta}{\alpha-1} \frac{\bmuf[\mu_n]}{\mu_n(\bmuf[\mu_n]) + \cte  + 1/(\alpha-1)}})} \eqsp, \quad n\in\nstar \eqsp.
\end{align*}

\subsection{Proof of \Cref{thm:Renyi}}
\label{sec:proofThmRenyi}

As we shall see, the proof can be adapted from the proof of \cite[Theorem 2]{daudel2020infinitedimensional}. For all $\mu \in \meas{1}(\Tset)$, we will use the notation
$$
\iteration^{AR}(\mu)(\rmd \theta) = \frac{\mu(\rmd \theta)  \exp \lrb{- \eta \frac{\bmuf(\theta)}{(\alpha-1)(\mu(\bmuf) + \cte) + 1} } }{\mu \lr{\exp \lrb{- \eta \frac{\bmuf}{(\alpha-1)(\mu_n(\bmuf) + \cte) + 1}}}} \eqsp
$$
to designate the one-step transition of the Renyi Descent algorithm. Note in passing that for all $\cte' \in \rset$, this definition can also be rewritten under the form
$$
\iteration^{AR}(\mu)(\rmd \theta) = \frac{\mu(\rmd \theta)  \exp \lrb{- \eta \frac{\bmuf(\theta)}{(\alpha-1)(\mu(\bmuf) + \cte) + 1}  + \cte' } }{\mu \lr{\exp \lrb{- \eta \frac{\bmuf}{(\alpha-1)(\mu_n(\bmuf) + \cte) + 1}  + \cte'}}} \eqsp.
$$

We also define
\begin{align} \label{eq:cteThm}
& \cteinf = \eta^{-1} \sup_{\theta \in \Tset, \mu \in \meas{1}(\Tset)} [(\alpha-1)(\bmuf(\theta) + \cte) + 1] \nonumber \\
& L = \eta^2 \sup_{v \in \Domain^{AR} }e^{-\eta v} \nonumber\\
& \ctesup = \sup_{v \in \Domain^{AR}} e^{\eta v} \nonumber\\
& \ctemono = \inf_{v \in \Domain^{AR}} \lrcb{ 1 - \eta (\alpha -1) (v-\cte')} \times \eta \inf_{v \in \Domain^{AR}}   e^{- \eta v} \eqsp.
\end{align}

\subsubsection{Recalling \cite[Lemma 5]{daudel2020infinitedimensional}}

Let $(\zeta,\mu)$ be a couple of probability measures where $\zeta$ is dominated by $\mu$ which we denote by $\zeta \preceq \mu$ and define
\begin{align}\label{eq:Aalpha}
  A_\alpha \eqdef \int_\Yset  \nu(\rmd y)  \int_\Tset \mu(\rmd \theta) k(\theta,y)  \falpha' \left( \frac{g(\theta)\mu k(y) }{p(y)} \right) \left[ 1 - g(\theta) \right]\eqsp,
\end{align}
where $g$ is the density of $\zeta$ w.r.t $\mu$, i.e. $\zeta(\rmd
\theta)=\mu(\rmd \theta) g(\theta)$. We recall \cite[Lemma 5]{daudel2020infinitedimensional} in \Cref{lem:fondam} below.

\begin{lem}{\cite[Lemma 5]{daudel2020infinitedimensional}}\label{lem:fondam}
Assume \ref{hyp:positive}. Then, for all $\mu,\zeta\in\meas{1}(\Tset)$ such that $\zeta \preceq \mu$ and $\Psif(\mu) < \infty$, we have
\begin{equation} \label{eq:bound:fondam}
A_\alpha \leq  \Psif(\mu) - \Psif(\zeta) \eqsp.
\end{equation}
Moreover, equality holds in \eqref{eq:bound:fondam} if and only if $\zeta=\mu$.
\end{lem}

\subsubsection{Adaptation of \cite[Theorem 1]{daudel2020infinitedimensional}}

\begin{lem}\label{thm:monotone}
  Assume \ref{hyp:positive} and \ref{hyp:gamma}. Let $\alpha \in \rset \setminus \lrcb{1}$, let $\cte$ be such that $(\alpha-1)\cte \geq 0$ and let $\mu \in \meas{1}(\Tset)$ be such that
\begin{align}\label{eq:admissRenyi}
0 <\mu \lrcb{\exp \lr{-\eta \frac{\bmuf + 1/(\alpha-1)}{(\alpha-1)(\mu(\bmuf) + \cte) + 1}}}<\infty \eqsp
\end{align}
 holds and $\Psif(\mu)<\infty$. Then, the two following assertions hold.
\begin{enumerate}[label=(\roman*)]
\item \label{item:mono1} We have  $\Psif \circ \iteration^{AR} (\mu) \leq \Psif(\mu)$.
\item \label{item:mono2} We have $\Psif \circ \iteration^{AR} (\mu) =\Psif(\mu)$ if and only if $\mu=\iteration^{AR} (\mu)$.
\end{enumerate}
\end{lem}

\begin{proof} The proof builds on the proof of \cite[Theorem 1]{daudel2020infinitedimensional} in the particular case $\alpha \in \rset \setminus \lrcb{1}$. Indeed, in this case,
\begin{align*}
A_\alpha &= \int_\Yset \nu(\rmd y) \int_\Tset \mu(\rmd \theta) k(\theta,y) \frac{1}{\alpha-1} \left[ \left( \frac{ g(\theta)\mu k(y)}{p(y)} \right)^{\alpha-1} - 1 \right] \left[ 1 - g(\theta) \right]  \\
&= \int_\Yset \nu(\rmd y) \int_\Tset \mu(\rmd \theta) k(\theta,y) \frac{1}{\alpha-1} \left( \frac{\mu k(y)}{p(y)} \right)^{\alpha-1} g(\theta)^{\alpha -1} \left[ 1 - g(\theta) \right]  \\
&= \int_\Tset \mu(\rmd \theta) \left[\bmuf(\theta) + \frac{1}{\alpha-1} \right] g(\theta)^{\alpha -1} \left[ 1 - g(\theta) \right] \eqsp.
\end{align*}
so that
\begin{align*}
A_\alpha &= [(\alpha-1) (\mu(\bmuf) + \cte) + 1] \times \int_\Tset \mu(\rmd \theta) \frac{\bmuf(\theta) + \frac{1}{\alpha-1}}{(\alpha-1) (\mu(\bmuf) + \cte) + 1} g(\theta)^{\alpha -1} \left[ 1 - g(\theta) \right]
\end{align*}
where $(\alpha-1) (\mu(\bmuf) + \cte) + 1 > 0$ under \ref{hyp:positive}. Set
$$
g=\tgamma \circ \lr{\frac{\bmuf + 1/(\alpha-1)}{(\alpha-1) (\mu(\bmuf) + \cte) + 1}}$$
where for all $v \in \Domain^{AR}$,
$$\tgamma (v) = \frac{e^{- \eta v}}{\mu \lrcb{ \exp \lr{-\eta  \frac{\bmuf + 1/(\alpha-1)}{(\alpha-1) (\mu(\bmuf) + \cte) + 1} - \eta \cte'}}}\eqsp.
$$
Finally, let us consider the probability space $(\Tset,\Tsigma,\mu)$ and let $V$ be the random variable
$$
V(\theta)= \frac{\bmuf(\theta) + 1/(\alpha-1)}{(\alpha-1) (\mu(\bmuf) + \cte) + 1} + \cte' \eqsp.
$$
Then, we have $\PE[1-\tgamma(V)] = 0$ and we can write
\begin{align}
A_\alpha &= [(\alpha-1) (\mu(\bmuf) + \cte) + 1] \times \PE[ (V-\cte') \tgamma^{\alpha-1}(V) (1-\tgamma(V))] \nonumber \\
 &= [(\alpha-1) (\mu(\bmuf) + \cte) + 1] \times \Cov( (V-\cte') \tgamma^{\alpha-1}(V), 1- \tgamma(V)) \eqsp. \label{eq:usefulNextLemma}
\end{align}
Under \ref{hyp:gamma} with $\alpha \in \rset \setminus \lrcb{1}$, $v \mapsto (v-\cte')\tgamma^{\alpha-1}(v)$ and $v \mapsto 1- \tgamma(v)$ are increasing on $\Domain^{AR}$ which implies $\Cov( V \tgamma^{\alpha-1}(V), 1- \tgamma(V)) \geq 0$ and thus $A_\alpha\geq 0$ since $(\alpha-1) (\mu(\bmuf) + \cte) + 1 > 0$.
\end{proof}

\subsubsection{Adaptation of \cite[Lemma 6]{daudel2020infinitedimensional}}

Consider the probability space $(\Tset,\Tsigma,\mu)$ and denote by $\Var_\mu$ the associated variance operator.

\begin{lem}\label{lem:mono:refined} Assume \ref{hyp:positive} and  \ref{hyp:gamma}. Let $\alpha \in \rset \setminus \lrcb{1}$, let $\cte$ be such that $(\alpha-1)\cte > 0$, and let $\mu\in\meas{1}(\Tset)$ be such that \eqref{eq:admissRenyi} holds and $\Psif(\mu)<\infty$. Then,
\begin{align}\label{eq:mono:refined}
 \frac{(\alpha-1)\cte \ctemono}{2} \Var_\mu \lr{\frac{\bmuf + 1/(\alpha -1)}{(\alpha-1) (\mu(\bmuf)+\cte)+1}} \leq \Psif(\mu) - \Psif\circ \iteration^{AR}(\mu) \eqsp,
\end{align}
where
$$
\ctemono \eqdef \inf_{v \in \Domain^{AR}} \lrcb{ 1 - \eta (\alpha -1)(v-\cte')} \times \inf_{v \in \Domain^{AR}}  \eta e^{- \eta v} \eqsp. $$
\end{lem}
\begin{proof} The proof of \Cref{lem:mono:refined} builds on the proof of \cite[Lemma 6]{daudel2020infinitedimensional}, which can be found in the supplementary material of \cite{daudel2020infinitedimensional}. Using \eqref{eq:usefulNextLemma} combined with the fact that under \ref{hyp:positive}, $(\alpha-1) (\mu(\bmuf) + \cte) + 1 > (\alpha-1) \cte > 0$
\begin{align*}
A_\alpha &= [(\alpha-1) (\mu(\bmuf) + \cte) + 1] \times \Cov( (V-\cte') \tgamma^{\alpha-1}(V), 1- \tgamma(V)) \\
& > (\alpha-1) \cte \times \Cov( (V-\cte') \tgamma^{\alpha-1}(V), 1- \tgamma(V))
\end{align*}
Furthermore,
\begin{align*}
\Cov( (V-\cte') & \tgamma^{\alpha-1}(V), 1- \tgamma(V)) \\
&= \frac{1}{2} \PE\left[((U-\cte')\tgamma^{\alpha-1}(U) - (V-\cte') \tgamma^{\alpha-1}(V))(-\tgamma(U) + \tgamma(V)) \right] \\
& = \frac{1}{2} \PE\left[\frac{(U-\cte') \tgamma^{\alpha-1}(U) - (V-\cte') \tgamma^{\alpha-1}(V)}{U - V}\frac{-\tgamma(U) + \tgamma(V)}{U-V} (U-V)^2 \right]\\
& \geq \frac{\ctemono}{2}\Var_\mu \lr{\frac{\bmuf + 1/(\alpha-1)}{(\alpha-1) (\mu(\bmuf)+\cte)+1}}
\end{align*}
and we thus obtain \eqref{eq:mono:refined}.
\end{proof}

\subsubsection{Adaptation of the proof of \cite[Theorem 2]{daudel2020infinitedimensional} to obtain \Cref{thm:Renyi}}

\begin{proof}[Proof of \Cref{thm:Renyi}] The proof of \Cref{thm:Renyi} builds on the proof of \cite[Theorem 2]{daudel2020infinitedimensional}, which can be found in the supplementary material of \cite{daudel2020infinitedimensional}. We prove the assertions successively.
\begin{enumerate}[label=(\roman*),wide=0pt, labelindent=\parindent]
 \item The proof of \ref{item:admiss1} simply consists in verifying that we can apply \Cref{thm:monotone}. For all $\mu \in \meas{1}(\Tset)$, \eqref{eq:admissRenyi} with $\mu = \mu_n$ holds for all $n \in \nstar$ by assumption on $\binfty$ and since at each step $n \in \nstar$, \Cref{thm:monotone} combined with $\Psif(\mu_n) < \infty$ implies that $\Psif(\mu_{n+1}) \leq \Psif(\mu_n) < \infty$, we obtain by induction that $(\Psif(\mu_n))_{n\in\nstar}$ is non-increasing.

 \item Let $n \in \nstar$, set $\Delta_n = \Psif(\mu_n) - \Psif(\mu^\star)$ and for all $\theta \in \Tset$,
 $V_n(\theta) = \frac{\bmuf[\mu_n](\theta) + \frac{1}{\alpha-1}}{(\alpha-1) (\mu_n(\bmuf[\mu_n]) + \cte) + 1} + \cte'$,
 such that $\rmd \mu_{n+1} \propto e^{- \eta V_n} \rmd \mu_n $.

  We first show that
\begin{align}\label{eq:delta_s}
\Delta_n \leq \cteinf \left[ \int_\Tset \log \left(\frac{\rmd \mu_{n+1}}{\rmd \mu_n} \right)\rmd \mu^\star  + \frac{L}{2} \Var_{\mu_n}(V_n) \ctesup \right] \eqsp.
\end{align}
The convexity of $\falpha$ implies that
\begin{align}\label{eq:rate1}
\Delta_n &\leq \int_\Tset \bmuf[\mu_n](\rmd \mu_n - \rmd \mu^\star) \\
&= \int_\Tset \lr{\bmuf[\mu_n] + \frac{1}{\alpha -1}}(\rmd \mu_n - \rmd \mu^\star) \nonumber \\
&= \frac{(\alpha-1)(\mu_n(\bmuf[\mu_n]) + \cte) + 1}{\eta} \int_\Tset (\mu_n(\eta V_n) - \eta V_n) \rmd \mu^\star \eqsp.
\end{align}
Then, noting that
$$
-\eta V_n = \log \mu_n \lr{e^{-\eta V_n}} + \log\lr{\frac{\rmd \mu_{n+1}}{\rmd \mu_n}}
$$
we deduce
\begin{align}\label{eq:rate3}
\Delta_n &\leq \cteinf \int_\Tset \left[\mu_n(\eta V_n) + \log \mu_n \lr{e^{-\eta V_n}} + \log\lr{\frac{\rmd \mu_{n+1}}{\rmd \mu_n}}  \right] \rmd \mu^\star \eqsp.
\end{align}
Since $v \mapsto e^{-\eta v}$ is $L$-smooth on $\Domain^{AR}$, for all $\theta \in \Tset$ and for all $n \in \nstar$ we can write
\begin{align*}
e^{- \eta V_n(\theta)} \leq e^{- \eta \mu_n(V_n)} + \eta e^{- \eta \mu_n(V_n)}(V_n(\theta) - \mu_n(V_n)) + \frac{L}{2} \left( V_n(\theta) - \mu_n(V_n) \right)^2
\end{align*}
which in turn implies
$$
\mu_n(e^{- \eta V_n}) \leq  e^{- \eta \mu_n(V_n)}  + \frac{L}{2} \Var_{\mu_n} \left( V_n \right) \eqsp.
$$
Finally, we obtain
$$
\log \mu_n(e^{- \eta V_n}) \leq \log e^{- \eta \mu_n(V_n)} + \log \left(1 +  \frac{L}{2} \frac{\Var_{\mu_n}(V_n) }{e^{- \eta \mu_n(V_n)}} \right) \eqsp.
$$
Using that $\log(1+u) \leq u$ when $u \geq 0$ and by definition of $\ctesup$, we deduce
$$
\log \mu_n(e^{- \eta V_n}) \leq - \eta \mu_n(V_n) + \frac{L}{2} \Var_{\mu_n}(V_n) \ctesup\eqsp,
$$
which combined with \eqref{eq:rate3} implies \eqref{eq:delta_s}. To conclude, we apply \Cref{lem:mono:refined} to $g = \frac{\rmd \mu_{n+1}}{\rmd \mu_n}$ and combining with \eqref{eq:delta_s}, we obtain
$$
\Delta_n \leq \cteinf \left[ \int_\Tset \log \left(\frac{\rmd \mu_{n+1}}{\rmd \mu_n} \right)\rmd \mu^\star  + \frac{L \ctesup }{\ctemono (\alpha-1)\cte} \left( \Delta_n - \Delta_{n+1} \right) \right] \eqsp,
$$
where by assumption $\ctemono$, $\cteinf$ and $\ctesup > 0$. As the r.h.s involves two telescopic sums, we deduce
\begin{align*}
\frac{1}{N} \sum_{n=1}^{N} \Psif(\mu_n) - \Psif(\mu^\star) \leq \frac{\cteinf}{N} \left[ KL\couple[\mu^\star][\mu_1] - KL\couple[\mu^\star][\mu_{N + 1}] + L\frac{ \ctesup}{\ctemono (\alpha-1)\cte} (\Delta_1 - \Delta_{N+1}) \right]
\end{align*}
and we recover \eqref{eq:rate} using \ref{item:admiss1}, that $KL\couple[\mu^\star][\mu_{N + 1}] \geq 0$ and that $\Delta_{N+1} \geq 0$.
\end{enumerate}
\end{proof}

\subsection{The Renyi Descent for mixture models: practical version}
\label{subsec:RDalgo}

The algorithm below provides one possible approximated version of the Renyi Descent algorithm, where we have set $\GammaAlpha(v) = e^{-\eta v}$ with $\eta > 0$.

\SetInd{0.6em}{-1.8em}
\begin{algorithm}[!h]
\caption{{\em Practical version of the Renyi Descent for mixture models}}
\label{algo:mixture:RD}
\textbf{Input:} $p$: measurable positive function, $K$: Markov transition kernel, $M$: number of samples, $\Theta = \{\theta_1, \ldots, \theta_J\} \subset \Tset$: parameter set, $\GammaAlpha(v) = e^{-\eta v}$ with $\eta > 0$, $N$: total number of iterations. \\
\textbf{Output:} Optimised weights $\lbd{}$. \\

Set $\lbd{} = [\lambda_{1, 1}, \ldots, \lambda_{J,1}]$.\\
\For{$n = 1 \ldots N$}{
\begin{enumerate}[label={}]
\item \underline{Sampling step} : \setlength{\parindent}{1cm} Draw independently $M$ samples $Y_{1}, \ldots, Y_{M}$ from $\mulbd k$.
\item \underline{Expectation step} : \setlength{\parindent}{1cm} Compute $\boldsymbol{B}_{\lbd{}} = (b'_{j})_{1 \leq j \leq J}$ where for all $j = 1 \ldots J$
\begin{align*}
b_{j} = \dfrac{1}{M} \sum \limits_{m=1}^M \frac{k(\theta_j, Y_{m})}{\mulbd k(Y_{m})}  \falpha'\left( \frac{\mulbd k(Y_{m})}{p(Y_{m})} \right)
\end{align*} 
and for all $j = 1 \ldots J$
$$
b'_j = \frac{ b_j}{(\alpha-1)(\sum_{\ell=1}^{J} b_\ell + \cte) +1} \eqsp
$$

\noindent and deduce $\boldsymbol{W}_{\lbd{}}  = (\lambda_j \GammaAlpha(b'_{j} + \cte'))_{1\leq j\leq J}$ and $w_{\lbd{}}  = \sum_{j=1}^J \lambda_j \GammaAlpha(b'_{j} + \cte')$.
\item \underline{Iteration step} : \setlength{\parindent}{1cm} Set
\begin{align*}
\lbd{} \leftarrow \frac{1}{w_{\lbd{}} } \boldsymbol{W}_{\lbd{}}
\end{align*}
\end{enumerate}
}
\end{algorithm}

\subsection{Alternative Exploration step in Algorithm \ref{algo:adaptive}}
\label{app:alter}

We present here several possible alternative choices of Exploration step in Algorithm \ref{algo:adaptive}, beyond the one we have made in \Cref{sec:numerical} and that is based on \cite{daudel2021monotonic}. Our goal here is not to discriminate between all of them, but to illustrate the generality of our approach.

\textbf{Gradient Descent.} One could use a Gradient Descent approach to optimise the mixture components parameters $\lrcb{\theta_{1, t+1}, \ldots, \theta_{J,t+1}}$ in the spirit of Renyi's $\alpha$-divergence gradient-based methods (e.g \cite{2015arXiv151103243H, 2016arXiv160202311L}) or $\alpha$-divergence gradient-based methods (e.g \cite{NIPS2017_6866, NIPS2017_14e422f0}).

\textbf{The particular case $\alpha \in [0,1)$.} Following \cite{daudel2021monotonic}, if we consider the specific case $\alpha \in [0,1)$ another possibility would be to set at time $t$: for all $j = 1 \ldots J$
\begin{align}\label{eq:updateEM}
\theta_{j,t+1} = \argmax_{\theta_j \in \Tset} \int_\Yset \respat[y] \log(k(\theta_{j}, y)) \nu(\rmd y)
\end{align}
where for all $y \in \Yset$,
\begin{align*}
\respat[y] = k(\theta_{j,t}, y) \lr{\frac{\mulbd{} k(y)}{  p(y)}}^{\alpha-1} \eqsp.
\end{align*}
Indeed, \cite{daudel2021monotonic} showed that the above update formulas for $\lrcb{\theta_{1,t+1}, \ldots, \theta_{J,t+1}}$ ensure a systematic decrease in the $\alpha$-divergence and they notably explained how these update formulas could even outperform typical Renyi's $\alpha$ / $\alpha$-divergence gradient-based approaches (we refer to \cite{daudel2021monotonic} for details).

Furthermore, in the particular case of $d$-dimensional Gaussian kernels with $k(\theta_{j,t}, y) = \mathcal{N}(y; m_{j,t}, \Sigma_{j,t})$ and where $\theta_{j,t} = (m_{j,t}, \Sigma_{j,t}) \in \Tset$ denotes the mean and covariance matrix of the $j$-th Gaussian component density, they obtained that the maximisation procedure \eqref{eq:updateEM} amounts to setting
\begin{align*}
\forall j = 1 \ldots J, \quad m_{j,t+1}&= \frac{\int_\Yset \respat[y] y ~  \nu(\rmd y)}{\int_\Yset \respat[y] \nu(\rmd y)} \\
\Sigma_{j,t+1} &= \frac{\int_\Yset  \respat[y] (y-m_{j,t})(y-m_{j,t})^T  \nu(\rmd y)}{ \int_\Yset  \respat[y] \nu(\rmd y)} \eqsp.
\end{align*}
These update formulas can then always be made feasible by resorting to Monte Carlo approximations and can be used as a valid Exploration step. If we were to focus on solely updating the means $(m_{j,t+1})_{1 \leq j \leq J}$, we could for example consider the Exploration step given by:
\begin{align*}
\forall j = 1 \ldots J, \quad \theta_{j, t+1} = m_{j,t+1}& = \frac{\sum_{m=1}^M \respa[Y' _{m}] \cdot Y' _{m} }{ \sum_{m=1}^M  \respa[Y' _{m}]}
\end{align*}
where the $M$ samples $(Y' _{m})_{1 \leq m \leq M}$ have been drawn independently from the proposal $\mu_{\lbd{}, \Theta}$ and where we have set
$$
\respa = \frac{k(\theta_{j,t}, y)}{\mu_{\lbd{}, \Theta} k(y)} \lr{\frac{\mu_{\lbd{}, \Theta} k(y)}{p(y)}}^{\alpha-1} \eqsp.
$$
We ran Algorithm \ref{algo:adaptive} over 100 replicates for this choice of Exploration step with $M \in \lrcb{100, 500}$ (and keeping the same target $p$, initial sampler $q_0$, and hyperparameters $N = 20$, $T=10$, $\eta = \eta_0/\sqrt{N}$ with $\eta_0 = 0.3$, $\alpha = 0.5$, $J = 100$, $\cte = 0.$ and $d = 16$ as those chosen in \Cref{sec:numerical}). The results when using the Power and the Renyi Descent as Exploitation steps can be visualised in the figure below.

\begin{figure}[ht!]
  \caption{Plotted is the average Variational Renyi bound for the Power Descent (PD) and the Renyi Descent (RD) in dimension $d = 16$ computed over 100 replicates with $\eta_0 = 0.3$ and $\alpha = 0.5$ and an increasing number of samples $M$.}
  \label{fig:ComparePDandRD2}
  \begin{center}
  \begin{tabular}{cc}
  \includegraphics[width=5.3cm]{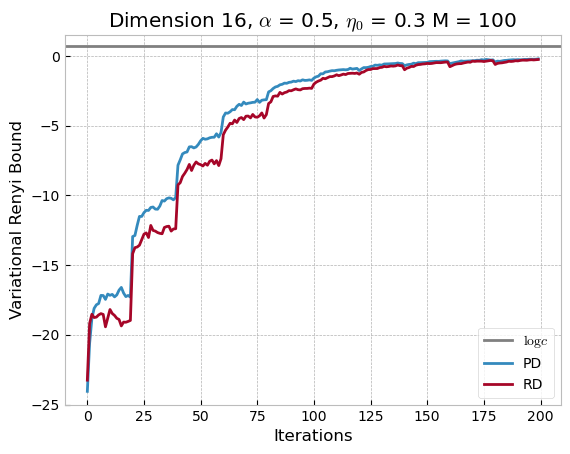} & \includegraphics[width=5.3cm]{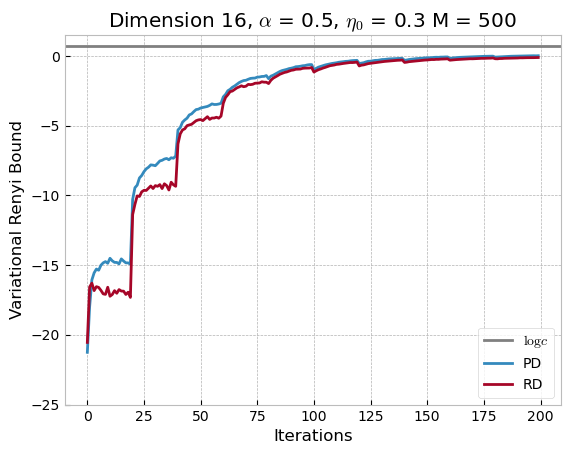}
  \end{tabular}
  \end{center}
\end{figure}

We then observe a similar behavior for the Power and the Renyi Descent, which illustrates the closeness between both algorithms, irrespective of the choice of the Exploration step. 

\end{document}